\documentclass[11pt]{article}
\usepackage{amssymb,amsmath,amsthm}
\numberwithin{equation}{section}
\newcommand{\FF}{\mathcal{F}}
\newcommand{\LL}{\mathcal{L}}

\newcommand{\OO}{\mathcal{O}}

\renewcommand{\SS}{\mathcal{S}}

\newcommand{\C}{\mathbb{C}}
\newcommand{\N}{\mathbb{N}}
\renewcommand{\P}{\mathbb{P}}
\newcommand{\R}{\mathbb{R}}

\newcommand{\Z}{\mathbb{Z}}

\newcommand{\dd}{\,{\rm d}}
\newcommand{\D}{{\rm d}}
\newcommand{\e}{\operatorname{e}}

\renewcommand{\div}{\operatorname{div}}
\newcommand{\curl}{\operatorname{curl}}

\newcommand{\meas}{\operatorname{meas}}
\renewcommand{\:}{\thinspace :}
\newcommand{\loc}{\mathrm{loc}}
\newcommand{\ul}{\mathrm{ul}}
\newcommand{\bu}{\mathrm{bu}}

\newcommand{\1}{\mathbf{1}}

\newcommand{\BMO}{\mathrm{BMO}}
\newcommand{\half}{{\textstyle\frac12}}
\newcommand{\rhoR}{\rho_{\scriptscriptstyle R}}
\newcommand{\rhoRt}{\rho_{\scriptscriptstyle R(t)}}

\newtheorem{theorem}{Theorem}[section]
\newtheorem{definition}[theorem]{Definition}
\newtheorem{proposition}[theorem]{Proposition}
\newtheorem{lemma}[theorem]{Lemma}
\newtheorem{corollary}[theorem]{Corollary}

\theoremstyle{definition}
\newtheorem{remark}[theorem]{Remark}

\newcommand{\QED}{\mbox{}\hfill$\Box$}
\numberwithin{equation}{section}

\begin{document}

\title{Infinite energy solutions of the two-dimensional \\ 
Navier-Stokes equations}

\author{{\bf Thierry Gallay} \\[1mm] 
Universit\'e de Grenoble I et CNRS\\
Institut Fourier, UMR 5582\\
100 rue des maths, B.P. 74\\
38402 Saint-Martin-d'H\`eres, France\\
{\tt Thierry.Gallay@ujf-grenoble.fr}}
\date{}

\maketitle

\begin{abstract}
These notes are based on a series of lectures delivered by the author
at the University of Toulouse in February 2014.  They are entirely
devoted to the initial value problem and the long-time behavior of
solutions for the two-dimensional incompressible Navier-Stokes
equations, in the particular case where the domain occupied by the
fluid is the whole plane~$\R^2$ and the velocity field is only assumed
to be bounded. In this context, local well-posedness is not difficult
to establish \cite{GIM}, and a priori estimates on the vorticity
distribution imply that all solutions are global and grow at most
exponentially in time \cite{GMS,ST}. Moreover, as was recently shown by
S.~Zelik, localized energy estimates can be used to obtain a much
better control on the uniformly local energy norm of the velocity
field \cite{Ze}.  The aim of these notes is to present, in an
explanatory and self-contained way, a simplified and optimized version 
of Zelik's argument which, in combination with a new formulation of the
Biot-Savart law for bounded vorticities, allows one to show that the
$L^\infty$ norm of the velocity field grows at most linearly in
time. The results do not rely on the viscous dissipation, and remain
therefore valid for the so-called ``Serfati solutions'' of the
two-dimensional Euler equations \cite{AKLN}. Finally, a recent work 
by S.~Slijep\v{c}evi\'{c} and the author shows that all solutions 
remain uniformly bounded in the viscous case if the velocity field and 
the pressure are periodic in one space direction \cite{GS1,GS2}. 
\end{abstract}

\tableofcontents

\section{Introduction}
\label{s1}

The aim of these notes is to present in a unified and rather
self-contained way a set of recent results by various authors which
give some valuable insight into the dynamics of the incompressible
Navier-Stokes equations in large or unbounded two-dimensional
domains. We must immediately point out that the restriction to the
two-dimensional case is more a technical necessity than a deliberate
choice: all questions that are discussed below would be equally
important and considerably more challenging for three-dimensional
fluids, but in the present state of affairs we simply do not know how
to address them mathematically. It should be mentioned, however, that
there exist situations where a two-dimensional approximation is
undoubtedly relevant for real fluids. This is the case, for instance, 
when the aspect ratio of the domain containing the fluid is very large, 
so that the motion in one space direction can be neglected under
certain conditions. Large-scale oceanic motion is a typical example 
that is good to keep in mind, although in that particular case a
realistic model should take into account additional effects such
as the Coriolis force, the wind forcing at the free surface, the 
topography of the bottom, or the energy dissipation in boundary 
layers. 

If the Navier-Stokes equations are considered in a smooth
two-dimensional domain $\Omega \subset \R^2$, with no-slip boundary
conditions, it is well known that there exists a unique global
solution in the energy space $L^2(\Omega)$ if the initial data have
finite kinetic energy. This fundamental result was first established
by J.~Leray in the particular situation where $\Omega$ is the whole
plane $\R^2$ \cite{Le1}. Bounded domains were also considered by
Leray, who proved local well-posedness in that case as well as global
well-posedness for small initial data \cite{Le2}. The restriction on
the size of the data was completely removed later \cite{La}, and the
existence proof was subsequently written in a nice functional-analytic
setting \cite{FK} which is applicable to essentially arbitrary domains
with smooth boundary, including for example exterior domains
\cite{KO}.  If no exterior force is exerted on the fluid, the kinetic
energy is a nonincreasing function of time that converges to zero as
$t \to \infty$, see \cite{Ma}. The rate of convergence is exponential
if $\Omega$ is bounded, due to the boundary conditions, and in the
unbounded case it depends on the localization properties of the
initial data \cite{Sch,Wi}. To conclude this brief survey, we also
mention that infinite-energy solutions can be considered in unbounded
two-dimensional domains, and may exhibit nontrivial long-time
asymptotics. For instance, in the whole plane $\R^2$, solutions with
integrable vorticity distribution but nonzero total circulation have
infinite kinetic energy, and converge toward nontrivial self-similar
solutions as $t \to \infty$ \cite{GW}.

The results mentioned above, and many others that were omitted, may
sometimes lead to the hasty conclusion that ``everything is known''
about the dynamics of the two-dimensional Navier-Stokes equations.
This is of course deeply incorrect, and a more careful thinking
reveals that even simple and natural questions still lack a
satisfactory answer. Here is a typical example, which motivates some
of the questions investigated in the present notes. Consider the free
evolution of a viscous incompressible fluid in a bounded
two-dimensional domain $\Omega\subset\R^2$, with no-slip boundary
conditions. We are interested in the situation where the domain is
very large compared to the length scale given by the kinematic
viscosity and the typical size of the velocity; in other words, the
{\em Reynolds number} of the flow is very high. If $D \subset \Omega$
is a small subdomain located far from the boundary $\partial\Omega$,
we are interested in estimating the kinetic energy of the fluid in the
observation domain $D$ at a given time $t$. That energy is certainly
smaller than the total kinetic energy of the fluid at time $t$, which
in turn is less than the same quantity at initial time, but such an
estimate is ridiculously non optimal. When sailing the ocean, nobody
expects that the total energy of the sea, or a substantial fraction of
it, could suddenly get concentrated in a small neighborhood of the
boat, and we certainly do not suggest this mechanism as a possible
explanation for the formation of rogue waves!  It is intuitively clear
that the energy in the subdomain $D$ at time $t$ should be essentially
independent of the size of the domain $\Omega$ and of the total
kinetic energy of the fluid; instead it should be possible to estimate
that quantity in terms of the size of $D$ and the initial energy {\em
  density} only, but to the author's knowledge no such result has been
established so far. In a more mathematical language, we are lacking
{\em uniformly local energy estimates} for the fluid velocity that
would hold uniformly in time and depend only on the initial energy
density. Such estimates would tell us how the energy can be
redistributed in the system, due to advection and diffusion, until it
is dissipated by the viscosity.

Since the questions we have just mentioned are independent of the 
size of the fluid domain and of the exact nature of the boundary
conditions, it seems reasonable to attack them first in the idealized 
situation where the fluid fills the whole plane $\R^2$ and the 
velocity field is merely bounded. We thus consider the Navier-Stokes 
equations
\begin{equation}\label{NSeq}
  \partial_t u + (u\cdot \nabla)u \,=\, \nu\Delta u -\frac{1}{\rho}
  \nabla \pi~, \qquad \div u \,=\, 0~,
\end{equation}
where the vector field $u(x,t) \in \R^2$ is the velocity of the fluid
at point $x \in \R^2$ and time $t \in \R_+$, and the scalar field
$\pi(x,t) \in \R$ is the pressure in the fluid at the same point. The
physical parameters in \eqref{NSeq} are the kinematic viscosity $\nu >
0$ and the fluid density $\rho > 0$, which are both assumed to be
constant. To eliminate the fluid density from \eqref{NSeq}, we
introduce the new function $p = \pi/\rho$, which we still call
(somewhat incorrectly) the ``pressure'' in the fluid. Many authors
also eliminate the kinematic viscosity by an appropriate rescaling,
but dimensionality checks then become more cumbersome, so we prefer 
keeping the parameter $\nu$. 

The first equation in \eqref{NSeq} corresponds to Newton's equation
for a fluid particle moving under the action of the pressure gradient
$-\nabla p$ and the internal friction $\nu\Delta u$, whereas the
relation $\div u = 0$ is the mathematical formulation of the
incompressibility of the fluid. The nonlinear advection term in
\eqref{NSeq} is due to the definition of the velocity field in the
Eulerian representation, which implies that the acceleration of a fluid
particle located at point $x \in \R^2$ it not $\partial_t u(x,t)$ but
$\partial_t u(x,t) +(u(x,t)\cdot\nabla)u(x,t)$. No evolution equation
for the pressure is needed, because $p$ can be expressed as a
nonlinear and nonlocal function of the velocity field $u$ by solving
the elliptic equation
\begin{equation}\label{pressure}
  -\Delta p \,=\, \div((u\cdot\nabla)u)~, 
\end{equation}
which is obtained by taking the divergence with respect to $x$ of the
first equation in \eqref{NSeq}. Note that \eqref{pressure} only
determines the pressure up to a harmonic function in $\R^2$, but if
the velocity field is bounded and divergence free one can show that
\eqref{pressure} has solution $p \in \BMO(\R^2)$ which is unique up to
an irrelevant additive constant. This is the canonical choice of the
pressure, which will always be made, albeit tacitly, in what follows.
Here $\BMO(\R^2)$ denotes the space of functions of bounded mean
oscillation in $\R^2$, see Section~\ref{ss2.2} below for a brief
presentation.  The interested reader should consult the monographs
\cite{CF,Li,MB,Te} for a careful derivation and a detailed discussion of
the model \eqref{NSeq}.

In most mathematical studies of the Navier-Stokes equations 
\eqref{NSeq}, it is assumed that the total (kinetic) energy of 
the fluid is finite\:
\begin{equation}\label{Edef}
  E(t) \,:=\, \frac{1}{2}\int_{\R^2} |u(x,t)|^2\dd x 
  \,<\, \infty~. 
\end{equation}
Strictly speaking the physical energy is $\rho E(t)$, but we use
definition \eqref{Edef} in agreement with our previous choice of
eliminating the density parameter $\rho$. It is important to realize
that $E(t)$ is a {\em Lyapunov function} for the flow of \eqref{NSeq},
because a formal calculation shows that
\begin{equation}\label{Eprime}
  \frac{\D }{\D t}E(t) \,=\, -\nu\int_{\R^2} |\nabla 
  u(x,t)|^2\dd x \,\le\, 0~.
\end{equation}
As a consequence, we have
\begin{equation}\label{EE}
  E(t) + \nu\int_0^t \!\int_{\R^2} |\nabla u(x,s)|^2\dd x \dd s
  \,=\, E(0)~, \qquad t \ge 0~.
\end{equation}
The energy equality \eqref{EE} plays a crucial role in Leray's 
construction of global solutions to the Navier-Stokes equations
in $\R^2$ \cite{Le1}. 

As was already mentioned, we consider in these notes the more general
situation where the velocity field is only assumed to be bounded. To
avoid inessential technical problems related to continuity at initial
time, we assume that $u$ belongs to the Banach space
\begin{equation}\label{Xdef}
  X \,=\, C_\bu(\R^2)^2 \,=\, \Bigl\{u : \R^2 \to \R^2 \,\Big|\,
  u \hbox{ is bounded and uniformly continuous}\Bigr\}~, 
\end{equation}
equipped with the uniform norm. If $u \in X$, the energy \eqref{Edef}
is infinite in general, but we can still consider the energy density
$e = \frac12 |u|^2$, which satisfies the following local version of
\eqref{Eprime}\:
\begin{equation}\label{EEloc}
  \partial_t e + \div\Bigl((p + e)u\Bigr) \,=\, 
  \nu \Delta e - \nu |\nabla u|^2~, \qquad x \in \R^2~, \quad
  t \ge 0~.
\end{equation}
Another important quantity is the vorticity $\omega = \curl u
= \partial_1 u_2 - \partial_2 u_1$, which evolves according 
to the simple advection-diffusion equation
\begin{equation}\label{omeq}
  \partial_t \omega + u\cdot\nabla \omega \,=\, \nu\Delta \omega~. 
\end{equation}
If the velocity field is bounded, one can apply the parabolic maximum
principle to \eqref{omeq} and prove that all $L^p$ norms of $\omega$
are Lyapunov functions for the flow of \eqref{omeq}. The case $p =
\infty$ is especially relevant for us, because if we assume that the
initial velocity $u_0$ belongs to $X$, standard parabolic smoothing
estimates imply that, for any positive time, the derivative $\nabla u$
is a bounded function on $\R^2$, see \eqref{smoothT} below. The
vorticity bound $\|\omega(\cdot,t)\|_{L^\infty}$ is therefore a finite
and nonincreasing function of time for all $t > 0$. We also mention
that, since $\div u = 0$ and $\curl u = \omega$, it is possible to
reconstruct the velocity field $u$ from the vorticity $\omega$, up to
an additive constant, by the Biot-Savart formula, see Section~\ref{A1}
for a detailed discussion. However, a uniform bound on the vorticity
does not allow to control the $L^\infty$ norm of the velocity field,
hence a priori estimates are not sufficient to prove that solutions of
\eqref{NSeq} stay uniformly bounded in time.

Global existence of solutions to the Navier-Stokes equations
\eqref{NSeq} in the space $X$ was first established by Giga, Matsui,
and Sawada \cite{GIM,GMS}. The proof in \cite{GMS} shows that the
$L^\infty$ norm of the velocity field cannot grow faster than a double
exponential as $t \to \infty$, but that pessimistic estimate was
subsequently improved by Sawada and Taniuchi \cite{ST} who obtained a
single exponential bound. These early results are summarized in 
our first statement\:

\begin{theorem}\label{thm1} {\bf \cite{GMS,ST}}
For any $u_0 \in X$ with $\div u_0 = 0$, the Navier-Stokes 
equations \eqref{NSeq} have a unique global (mild) solution 
$u \in C^0([0,+\infty),X)$ with initial data $u_0$. Moreover, 
if the initial vorticity $\omega_0 = \curl u_0$ is bounded, 
we have the estimate
\begin{equation}\label{thm1est}
  \|u(\cdot,t)\|_{L^\infty} \,\le\, K_0 \|u_0\|_{L^\infty}\,\exp
  \Bigl(K_0\,\|\omega_0\|_{L^\infty} t\Bigr)~, \qquad t \ge 0~,
\end{equation}
where $K_0 \ge 1$ is a universal constant. 
\end{theorem}

In Theorem~\ref{thm1}, a {\em mild} solution refers to a solution of
the integral equation associated with \eqref{NSeq}, see
Section~\ref{ss2.4} below for details. The assumption that the initial
vorticity be bounded is only needed to derive the nice estimate
\eqref{thm1est}, which does not depend on the viscosity parameter. 
If we only suppose that $u_0 \in X$, $\div u_0 = 0$, and $u_0 
\not\equiv 0$, the local existence theory shows that 
\begin{equation}\label{smoothT}
  \sup_{0 \le t \le T}\|u(\cdot,t)\|_{L^\infty} + \sup_{0 < t \le T}
  (\nu t)^{1/2}\|\nabla u(\cdot,t)\|_{L^\infty} \,\le\, K_1 
  \|u_0\|_{L^\infty}\,, \quad \hbox{for}~\, 
  T \,=\, \frac{\nu}{K_1^2\|u_0\|_{L^\infty}^2}~,
\end{equation}
where $K_1 \ge 1$ is a universal constant, see Section~\ref{ss2.4}.
It follows in particular from \eqref{smoothT} that
$\|\omega(\cdot,T)\|_{L^\infty} \le 2\|\nabla u(\cdot,T)\|_{L^\infty}
\le 2K_1^2 \nu^{-1}\|u_0\|_{L^\infty}^2$, so if we use \eqref{smoothT}
for $t \in [0,T]$ and \eqref{thm1est} for $t \ge T$ we obtain a bound
of the form
\begin{equation}\label{thm1estbis}
  \|u(\cdot,t)\|_{L^\infty} \,\le\, K_2 \|u_0\|_{L^\infty}\,\exp
  \Bigl(K_2\,\nu^{-1}\|u_0\|_{L^\infty}^2 t\Bigr)~, \qquad t \ge 0~,
\end{equation}
for some universal constant $K_2 \ge 1$. Estimate \eqref{thm1estbis}
holds for all $u_0 \in X$ with $\div u_0 = 0$, but the right-hand
side depends explicitly on the viscosity parameter $\nu$. 

There are reasons to believe that the exponential upper bound
\eqref{thm1est} is the best one can obtain if one only uses the a
priori estimates given by the vorticity equation. However, as was
shown recently by S.~Zelik \cite{Ze}, the above results can be
improved in a spectacular way if one also exploits the local
dissipation law \eqref{EEloc}, which asserts that no energy is created
inside the system. The work of Zelik is devoted to a more general
Navier-Stokes system, which includes an additional linear damping term
and an external force, but in the particular case of equation
\eqref{NSeq} a slight extension of the results of \cite{Ze} gives the
following statement\:

\begin{theorem}\label{thm2} {\bf \cite[revisited]{Ze}}
If $u_0 \in X$, $\div u_0 = 0$, and $\omega_0 = \curl u_0 
\in L^\infty(\R^2)$, the solution of the Navier-Stokes equations 
\eqref{NSeq} given by Theorem~\ref{thm1} satisfies
\begin{equation}\label{thm2est}
  \|u(\cdot,t)\|_{L^\infty} \,\le\, K_3 \|u_0\|_{L^\infty}\Bigl(1 + 
  \|\omega_0\|_{L^\infty} t\Bigr)~, \qquad t \ge 0~,
\end{equation}
where $K_3 \ge 1$ is a universal constant. 
\end{theorem}

Estimate \eqref{thm2est} is clearly superior to \eqref{thm1est}, 
because it shows that the $L^\infty$ norm of the velocity field grows
at most linearly as $t \to \infty$. As before, if we do not assume 
that $\omega_0 \in L^\infty(\R^2)$, we can use \eqref{smoothT} for 
short times to prove that the bound \eqref{thm2est} remains valid if 
$\|\omega_0\|_{L^\infty}$ is replaced by $2K_1^2\nu^{-1} \|u_0\|_{L^\infty}^2$ 
in the right-hand side. We thus find
\begin{equation}\label{thm2estbis}
  \|u(\cdot,t)\|_{L^\infty} \,\le\, K_4 \|u_0\|_{L^\infty}\biggl(1 + 
  \frac{\|u_0\|_{L^\infty}^2t}{\nu}\biggr)~, \qquad t \ge 0~,
\end{equation}
for some universal constant $K_4 \ge 1$. 

The strategy of the proof of Theorem~\ref{thm2} in \cite{Ze} can be
roughly explained as follows. Suppose that we want to control the
solution of \eqref{NSeq} given by Theorem~\ref{thm1} on some large
time interval $[0,T]$. A natural idea is to compute, for $t \in
[0,T]$, the amount of energy contained in the ball of radius $R > 0$
centered at $x \in \R^2$\:
\begin{equation}\label{ERdef}
  E_R(x,t) \,=\, \frac12 \int_{B_x^R}|u(y,t)|^2 \dd y~, \qquad 
  \hbox{where}\quad B_x^R \,=\, \Bigl\{y \in \R^2\,\Big|\, 
  |y-x| \le R\Bigr\}~.
\end{equation}
Although the Navier-Stokes equations are dissipative, it is clear that
$E_R(x,t)$ is not necessarily a decreasing function of time, because
energy may enter the ball $B_x^R$ through the boundary due to the
advection term $\div((p+e)u)$ and the diffusion term $\nu \Delta e$ in
\eqref{EEloc}. However, the key observation is that these energy
fluxes become relatively negligible when the radius $R$ is taken
sufficiently large.  Indeed, since the velocity field $u(\cdot,t)$ is
bounded on $\R^2$ for any $t \in [0,T]$, we expect that for large $R$
the energy $E_R(x,t)$ will be proportional to the area of the ball
$B_x^R$, which is $\pi R^2$, whereas the flux terms will be
proportional to the length of the boundary $\partial B_x^R$, which is
$2\pi R$. This suggests that taking $R$ sufficiently large, depending
on $T$, may help controlling the relative contribution of the energy
entering the ball $B_x^R$ through the boundary. As a matter of fact,
S.~Zelik proved in \cite{Ze} that there exists a universal constant
$K_5 \ge 1$ such that
\begin{equation}\label{ulER}
  \sup_{t \in [0,T]}\,\sup_{x \in \R^2} \frac{1}{\pi R^2} 
  \int_{B_x^R}|u(y,t)|^2 \dd y \,\le\, K_5 \sup_{x \in \R^2} 
  \frac{1}{\pi R^2} \int_{B_x^R}|u_0(y)|^2 \dd y \,\le\, 
  K_5 \|u_0\|_{L^\infty}^2~, 
\end{equation}
provided $R$ is taken sufficiently large, depending on $T$. More
precisely, we shall see in Section~\ref{ss3.3} below that one can take
$R = \max\{R_0\,,\,C\sqrt{\nu T}\,,\,C\|u_0\|_{L^\infty}
\|\omega_0\|_{L^\infty}T^2\}$, where $C > 0$ is a universal constant
and $R_0 = \|u_0\|_{L^\infty}/\|\omega_0\|_{L^\infty}$. Estimate
\eqref{ulER} is an example of {\em uniformly local energy estimate}
for the Navier-Stokes equations, because the quantity it involves is
equivalent to the square of the norm of $u$ in the uniformly local
Lebesgue space $L^2_\ul(\R^2)$, see Section~\ref{ss3.1} for an
introduction to these spaces. It is clear that \eqref{ulER} is optimal
in the sense that, if the initial velocity $u_0$ is a nonzero
constant, then $u(\cdot,t) = u_0$ for all $t \ge 0$ and \eqref{ulER}
becomes an equality if $K_5 = 1$. What may not be optimal is 
the dependence of the radius $R$ upon the observation time $T$, 
namely $R = \OO(T^2)$ as $T \to \infty$. If we had \eqref{ulER}
for a smaller value of $R$, this would improve inequality
\eqref{thm2est}, because as we shall see in Section~\ref{ss3.4} the
right-hand side of \eqref{thm2est} behaves like $\|u_0\|_{L^\infty}^{1/2} 
\|\omega_0\|_{L^\infty}^{1/2}R(t)^{1/2}$ as $t \to \infty$.

It is worth emphasizing that estimates \eqref{thm1est} and
\eqref{thm2est} do not involve the viscosity parameter $\nu$, and thus
do not rely on energy dissipation in the system. Passing to the limit
as $\nu \to 0$, they remain valid for global solutions of the Euler
equations in $\R^2$ with bounded velocity and vorticity. Such
solutions were recently studied by Ambrose, Kelliher, Lopes Filho, and
Nussenzveig Lopes in \cite{AKLN}, following an earlier work by
Ph. Serfati \cite{Se}, see also \cite{Ke} for further improvements.
Existence can be proved by an approximation argument, which is quite
different from the simple existence proof presented in
Section~\ref{s2} for the Navier-Stokes equations, but once global
solutions have been constructed the bounds \eqref{thm1est},
\eqref{thm2est} can be established just as in the viscous case, see
also \cite{CZ}. On the other hand, if one does use energy dissipation
when $\nu > 0$, it is possible to obtain the following uniformly local
enstrophy estimate\:
\begin{equation}\label{ulom}
   \,\sup_{x \in \R^2} \frac{1}{\pi R^2} \int_{B_x^R}|\omega(y,t)|^2 \dd y 
  \,\le\, K_6\,\frac{\|u_0\|_{L^\infty}^2}{\nu t}~, \qquad 0 < t \le T~, 
\end{equation}
where $R = R(T)$ is as in \eqref{ulER} and $K_6 > 0$ is a universal
constant. Of course, if the initial vorticity is bounded, the
left-hand side of \eqref{ulom} is also smaller than 
$\|\omega_0\|_{L^\infty}^2$ by the maximum principle. 
Estimate \eqref{ulom} shows that a suitable average of the 
vorticity distribution converges to zero like $t^{-1/2}$ as 
$t \to \infty$. This strongly suggests that the long-time behavior
of solutions to \eqref{NSeq} should be governed by irrotational 
flows, although no precise statement is available so far. 

Theorem~\ref{thm2} is the best we can do without further assumptions
on the initial data. Since the right-hand side of \eqref{thm2est}
still depends on time, although in a rather mild way, we do not have a
satisfactory answer yet to the original question of estimating the
energy of a solution of \eqref{NSeq} in an observation domain $D
\subset \R^2$ in terms of the initial energy density only. There is no
reason to believe that the linear time dependence in \eqref{thm2est}
is sharp, and to the author's knowledge there is no example of 
a solution to the Navier-Stokes equations \eqref{NSeq} for which
the $L^\infty$ norm of the velocity field grows unboundedly in time. 
However, we believe that genuinely new ideas are needed to
improve estimate \eqref{thm2est}. 

To conclude this introduction, we briefly present an interesting
particular case where the conclusion of Theorem~\ref{thm2} can be
substantially strengthened. Following \cite{AM, GS1, GS2}, we consider
the Navier-Stokes equations \eqref{NSeq} in the infinite strip
$\Omega_L = \R \times [0,L]$, with periodic boundary conditions. 
Equivalently, we restrict ourselves to solutions of \eqref{NSeq} 
in $\R^2$ for which the velocity field $u(x,t)$ and the pressure 
$p(x,t)$ are periodic of period $L > 0$ in one space direction, 
which is chosen to be the second coordinate axis. We denote by $X_L$ 
the set of all $u \in X$ such that $u(x_1,x_2) = u(x_1,x_2+L)$ for 
all $x = (x_1,x_2) \in \R^2$. If $u \in X_L$ is divergence free, 
one can show that the elliptic equation \eqref{pressure} has a 
bounded solution which is $L$-periodic with respect to the 
second coordinate $x_2$, and that this solution is unique up 
to an additive constant. This is the canonical definition of 
the pressure in the present context, which agrees with the 
choice made in Theorems~\ref{thm1} and \ref{thm2}. 
We are now in position to state our last result\:

\begin{theorem}\label{thm3} {\bf \cite{GS2}}
For any $u_0 \in X_L$ with $\div u_0 = 0$, the Navier-Stokes 
equations \eqref{NSeq} have a unique global (mild) solution 
$u \in C^0([0,+\infty),X_L)$ with initial data $u_0$. Moreover,
we have the estimate
\begin{equation}\label{thm3est}
  \|u(\cdot,t)\|_{L^\infty} + (\nu t)^{1/2}\|\omega(\cdot,t)
  \|_{L^\infty} \,\le\, K_7 \|u_0\|_{L^\infty}(1+ R_u^5)~, 
  \qquad t > 0~,
\end{equation}
where $R_u = \nu^{-1}L\|u_0\|_{L^\infty}$ is the initial Reynolds 
number and $K_7 > 0$ is a universal constant. 
\end{theorem}

The conclusion of Theorem~\ref{thm3} is obviously much stronger than
that of Theorem~\ref{thm2}. First, the right-hand side of
\eqref{thm3est} does not depend on time, so that the velocity field
$u(\cdot,t)$ is uniformly bounded for all times. This is the result
that we were not able to prove in the general case. Next, the
vorticity distribution $\omega(\cdot,t)$ converges uniformly to zero
as $t \to \infty$, at the optimal rate $\OO(t^{-1/2})$ which is the
same as for the linear heat equation. This is clearly compatible with
\eqref{ulom}, but the estimate is now much more precise. In addition,
the proof of Theorem~\ref{thm3} given in \cite{GS2} provides detailed
informations on the long-time behavior of the solutions, which are
shown to converge exponentially in time to a shear flow governed by a
linear advection-diffusion equation on the real line $\R$. These very
strong conclusions are obtained using, in particular, the crucial
observation made in \cite{AM} that the Biot-Savart law is more
powerful when periodicity is assumed in one space direction. Indeed, a
uniform bound on the vorticity $\omega$ allows to control $L^\infty$
norm of the velocity field, except for the quantity $m(x_1,t) = 
L^{-1}\int_0^L u_2(x_1,x_2,t)\dd x_2$, which represents the
average of the second component of the velocity over one period. 
We observe, however, that the right-hand side of \eqref{thm3est} depends
on the viscosity parameter $\nu$, through the initial Reynolds number
$R_u$, and does not have a finite limit as $\nu \to 0$. As a matter of
fact, energy dissipation is an essential ingredient in the proof of
Theorem~\ref{thm3}.

The rest of these notes is organized as follows. Section~\ref{s2} is
entirely devoted to the proof of Theorem~\ref{thm1}. For the reader's
convenience, we first recall well known properties of the heat
semigroup in the space \eqref{Xdef}, we study the elliptic equation
\eqref{pressure}, and we define the Leray-Hopf projection which allows
us to eliminate the pressure from \eqref{NSeq}. We then establish
local existence of solutions in $X$ by applying a fixed point argument
to the integral equation associated with \eqref{NSeq}. Finally, we
prove global existence and obtain the exponential bound
\eqref{thm1est} using a nice Fourier-splitting argument borrowed from
\cite{ST}. In Section~\ref{s3}, we develop the uniformly local energy
estimates which are the key ingredient in the proof of
Theorem~\ref{thm2}. We first introduce the uniformly local Lebesgue
spaces, and specify a class of weight functions that can be used to
construct equivalent norms. Then, as a warm-up, we apply uniformly
local $L^2$ and $L^p$ estimates to solutions of the linear heat
equation. The core of the proof is Section~\ref{ss3.3} where,
following the approach of Zelik \cite{Ze}, we use uniformly local
energy estimates to control the solutions of the Navier-Stokes
equations in $\R^2$. This gives estimate \eqref{ulER}, and it is then
relatively simple to deduce the upper bound \eqref{thm2est} as
well as the enstrophy estimate \eqref{ulom}. The final section is an
appendix, where important auxiliary results are established.  We first
discuss the Biot-Savart formula, which allows us to reconstruct the
velocity field from the vorticity up to an additive constant. We also
establish a new representation formula for the pressure, which can be
expressed as an absolutely convergent integral involving the velocity
field and the vorticity. Finally miscellaneous notations and results
are collected in the last subsection, for easy reference. 

\medskip\noindent {\bf Disclaimer.} The present text is a set of
lecture notes, not an original research article. Most of the results
presented here have already been published elsewhere, and are not due
to the author. In particular, the proof of Theorem~\ref{thm1} in
Section~\ref{s2} is entirely taken from \cite{GIM,GMS,ST}, and the
preliminary material collected in Sections~\ref{ss2.1}--\ref{ss2.3}
can be found in many textbooks. Section~3 is a little bit more
original, although the statement and the proof of Theorem~\ref{thm2}
are taken from the work of Zelik \cite{Ze} with relatively minor
modifications, see also \cite{CZ} for recent improvements in the same
direction.  Lemma~\ref{rholem} is apparently new, and gives a
characterization of admissible weights which is substantially more
general than what can be found in the literature, see
e.g. \cite[Definition~4.1]{ARCD}.  Also, the way we treat the pressure
in Section~\ref{ss3.3} differs notably from \cite{Ze} and simplifies
somewhat the argument by avoiding the use of delicate interpolation
inequalities established in the appendices of \cite{Ze} and
\cite{CZ}. The linear bound \eqref{thm2est} does not appear explicitly
in \cite{Ze}, but follows quite easily from the uniformly local energy
estimate \eqref{ulER} and the a priori bound on the vorticity, see
\cite{CZ}. Finally, the Biot-Savart formula and the representation of
the pressure given in Section~\ref{s4} are apparently new, although
the recent work \cite{Ke} contains several interesting results in the
same spirit.

\medskip\noindent
{\bf Acknowledgements.} These notes are based on a six-hour
mini-course given during the thematic week ``Deterministic and
stochastic Navier-Stokes equations'' which was held on February 17-21, 
2014 at the University of Toulouse (France). The author warmly thanks
Violaine Roussier and Patrick Cattiaux for the invitation and the
perfect organizational work. Financial support by the LabEx CIMI and
the ANR project Dyficolti ANR-13-BS01-0003-01 is also gratefully
acknowledged. Finally the author thanks Sergey Zelik for valuable
discussions concerning the work \cite{Ze}, which served as a 
basis for the present notes. 

\section{The Cauchy problem with bounded initial data}
\label{s2}

In this section we study the Cauchy problem for the Navier-Stokes
equations
\begin{equation}\label{2DNS}
  \partial_t u + (u\cdot\nabla)u \,=\, \nu\Delta u - \nabla p~, 
  \qquad \div u \,=\, 0~,
\end{equation}
in the whole plane $\R^2$, with bounded initial data. We thus
assume that the velocity field $u = (u_1,u_2)$ belongs to the 
Banach space $X$ defined in \eqref{Xdef}, which is equipped with 
the uniform norm
\[
  \|u\|_{L^\infty} \,=\, \sup_{x \in \R^2}|u(x)|~, \qquad
  \hbox{where}\quad |u| \,=\, (u_1^2 + u_2^2)^{1/2}~.
\]
Our first goal is to reformulate the Navier-Stokes equations 
\eqref{2DNS} as an integral equation in $X$. This requires 
three preliminary steps, which are performed in 
Sections~\ref{ss2.1}-\ref{ss2.3}. 

\subsection{The heat semigroup on $C_\bu(\R^2)$}
\label{ss2.1}

Let $\LL(X)$ be the space of all bounded linear operators
on $X$. For any $t > 0$, we denote by $S(t) \in \LL(X)$ 
the linear operator defined, for all $u_0 \in X$, by the formula
\begin{equation}\label{heatkernel}
  \Bigl(S(t)u_0\Bigr)(x) \,=\, \frac{1}{4\pi t}\int_{\R^2} 
  \e^{-|x-y|^2/(4t)}u_0(y)\dd y~, \qquad x \in \R^2~. 
\end{equation}
We also set $S(0) = \1$ (the identity map). The family 
$\{S(t)\}_{t \ge 0}$ has the following properties, which are 
well known and easy to verify \cite[Section 2.3]{Ev}. 

\begin{enumerate}

\item 
For any $u_0 \in X$, one has $S(t)u_0 \in X$ for any $t \ge 0$
and $\|S(t)u_0\|_{L^\infty} \le \|u_0\|_{L^\infty}$.  That bound
holds because the heat kernel in \eqref{heatkernel} is positive and
normalized so that
\[
  \frac{1}{4\pi t}\int_{\R^2} \e^{-\frac{|x|^2}{4t}}\dd x \,=\, 
  1~, \qquad \hbox{for any } t > 0~. 
\]
\item One has $S(t_1 + t_2) = S(t_1)S(t_2)$ for all $t_1, t_2 \ge 0$. 
If both $t_1, t_2$ are positive, this follows from the identity
\[
  \frac{1}{4\pi}\int_{\R^2} \e^{-\frac{|x-y|^2}{4t_1}}\,\e^{-\frac{|y|^2}{4t_2}}
  \dd y \,=\, \frac{t_1 t_2}{t_1 + t_2}\e^{-\frac{|x|^2}{4(t_1+t_2)}}~,
  \qquad x \in \R^2~, 
\]
which can be established by a direct calculation, or by using the 
Fourier transform to compute the convolution product in the 
left-hand side. 

\item For any $u_0 \in X$, the map $t \mapsto S(t)u_0$ is continuous
from $[0,\infty)$ into $X$. More generally, for any $u_0 \in L^\infty(\R^2)^2$, 
one can verify that $t \mapsto S(t)u_0$ is continuous from 
$(0,\infty)$ into $X$, but right-continuity at $t = 0$ holds only
if $u_0 \in X$. 

\item If $u_0 \in X$ and if we set $u(x,t) = (S(t)u_0)(x)$ for 
$x \in \R^2$ and $t \ge 0$, then $u$ is smooth for $t > 0$ and 
satisfies the heat equation
\begin{equation}\label{heateq}
  \left\{\begin{array}{l} \partial_t u(x,t) \,=\, \Delta u(x,t)~, \\
  ~~\,u(x,0) \,=\, u_0(x)~,\end{array} \qquad 
  \begin{array}{l}x \in \R^2\,,\quad t > 0\,,\\ x \in \R^2\,.\end{array}
  \right.
\end{equation}
In fact $u$ is the unique bounded solution of \eqref{heateq}. 

\item For any $u_0 \in X$ and any multi-index $\alpha = (\alpha_1,
\alpha_2) \in \N^2$, there exists a constant $C > 0$ such that
\begin{equation}\label{heatder}
  \|\partial^\alpha S(t)u_0\|_{L^\infty} \,\le\, \frac{C}{t^{|\alpha|/2}}\,
  \|u_0\|_{L^\infty}~, \qquad \hbox{for all }t > 0~,
\end{equation}
where $\partial^\alpha = \partial_1^{\alpha_1}\partial_2^{\alpha_2}$ and
$|\alpha| = \alpha_1 + \alpha_2$. In particular $\|\nabla S(t)
u_0\|_{L^\infty} \le C t^{-1/2}\|u_0\|_{L^\infty}$. 
\end{enumerate}

\noindent
Properties~1--3 above can be summarized by saying that the 
family $\{S(t)\}_{t\ge 0}$ is a {\em strongly continuous semigroup 
of contractions in} $X$, see \cite{EN,Pa}. Property~4 implies that 
the Laplacian operator is the {\em generator} of the heat semigroup. 
Finally, the smoothing estimates \eqref{heatder} are related to the 
{\em analyticity} of the semigroup $\{S(t)\}_{t\ge 0}$ in $X$. 

\subsection{Determination of the pressure}
\label{ss2.2}

Applying the divergence operator to the first equation in \eqref{2DNS}, 
we obtain the elliptic equation
\begin{equation}\label{peq}
  -\Delta p(x) \,=\, \div\Bigl((u(x)\cdot\nabla)u(x)\Bigr)~, 
  \qquad x \in \R^2~,
\end{equation}
which determines the pressure $p$ up to a harmonic function
on $\R^2$. To construct a particular solution we observe that, 
since $\div u = 0$, we can write \eqref{peq} in the equivalent form
\[
  -\Delta p(x) \,=\, \sum_{k,\ell = 1}^2 \partial_k \partial_\ell 
  \Bigl(u_k(x) u_\ell(x)\Bigr)~, \qquad x \in \R^2~.
\]
If we take the Fourier transform of both sides and use the 
conventions specified in Section~\ref{A3}, we thus find
\[
  |\xi|^2 \hat p(\xi) \,=\, \sum_{k,\ell = 1}^2 (i\xi_k)(i\xi_\ell)
  \widehat{u_k u_\ell}(\xi)~, \qquad \hbox{hence}\quad
  \hat p(\xi) \,=\, \sum_{k,\ell = 1}^2 \frac{i\xi_k}{|\xi|}
  \frac{i\xi_\ell}{|\xi|}\,\widehat{u_k u_\ell}(\xi)~, 
\]
where equality holds in the space of tempered distributions 
$\SS'(\R^2)$. This gives (at least formally) the following 
solution to \eqref{peq}
\begin{equation}\label{pdef}
  p \,=\, \sum_{k,\ell = 1}^2 R_k R_\ell (u_k u_\ell)~,
\end{equation}
where $R_1, R_2$ are the {\em Riesz transforms} on $\R^2$, 
namely the linear operators defined as Fourier multipliers 
through the formulas
\begin{equation}\label{Rieszdef}
  \widehat{R_k f}(\xi) \,=\, \frac{i\xi_k}{|\xi|}\,\hat f(\xi)~, \qquad
  k = 1,2~, \qquad \xi \in \R^2\setminus\{0\}~. 
\end{equation}
Here $f \in \SS(\R^2)$ is an arbitrary test function. In ordinary space, 
the Riesz transforms are {\em singular integral operators} of the form
\[
  (R_k f)(x) \,=\, -\frac{1}{2\pi} \lim_{\epsilon \to 0}
  \int_{|y| \ge \epsilon} f(x-y)\frac{y_k}{|y|^3}\dd y~, \qquad
  k = 1,2~, \qquad x \in \R^2~.
\]
Using the Calder\'on-Zygmund theory \cite[Chapter I]{St}, one 
can prove that the Riesz transforms define bounded linear operators 
on $L^p(\R^2)$ for $p \in (1,\infty)$\:
\begin{equation}\label{RieszLp}
  \|R_k f\|_{L^p} \,\le\, C_p \|f\|_{L^p}~, \qquad k = 1,2~, 
  \qquad 1 < p < \infty~.
\end{equation}
Unfortunately, estimate \eqref{RieszLp} fails both for $p = 1$ 
and $p = \infty$. In particular, if $f \in L^\infty(\R^2)$, the
Riesz transform $R_k f$ is not a bounded function in general, but 
a function of {\em bounded mean oscillation} in the sense 
of the following definition. 

\begin{definition}\label{BMOdef}
A locally integrable function $f$ on $\R^2$ belongs to $\BMO(\R^2)$ 
if there exists $A \ge 0$ such that, for any ball $B \subset \R^2$
with nonzero Lebesgue measure $|B|$, one has
\begin{equation}\label{BMOineq}
  \frac{1}{|B|} \int_B |f(x) - f_B|\dd x \,\le\, A~, \qquad 
  \hbox{where}\quad f_B \,=\, \frac{1}{|B|}\int_B f(x)\dd x~.
\end{equation}
If $f \in \BMO(\R^2)$, the smallest bound $A$ in \eqref{BMOineq} is 
denoted by $\|f\|_{\BMO}$. 
\end{definition}

If $f \in L^\infty(\R^2)$, it is clear that $f \in \BMO(\R^2)$ and
$\|f\|_{\BMO} \le 2 \|f\|_{L^\infty}$. However, the space $\BMO(\R^2)$
is strictly larger than $L^\infty(\R^2)$. For instance, if $f(x) =
\log|x|$, then $f$ has bounded mean oscillation \cite[\S IV.1.1]{St},
but $f$ is obviously unbounded. It is also clear that adding a
constant to $f$ does not alter the quantity $\|f\|_{\BMO}$ which,  
therefore, is not a norm. However, one can show that $\|\cdot\|_{\BMO}$
defines a norm on the quotient space of $\BMO(\R^2)$ modulo the space
of constant functions, which becomes in this way a Banach space. We
refer to \cite[Chapter IV]{St} for a comprehensive study of functions
of bounded mean oscillation.

Returning to Riesz transforms, we mention the important fact that
$R_1, R_2$ can be extended to bounded linear operators from
$L^\infty(\R^2)$ into $\BMO(\R^2)$, and even from $\BMO(\R^2)$ into
itself, see e.g. \cite[Section~VII.4]{Me} for more general results
implying that particular one. We point out that these extensions have
the property that $R_1, R_2$ vanish on constant functions. As a
consequence, if $u \in X$ is divergence free, the formula \eqref{pdef}
makes sense and defines a function $p \in \BMO(\R^2)$, which satisfies
the elliptic equation \eqref{peq} in the sense of distributions. Thus
we have proved\:

\begin{lemma}\label{plem}
If $u \in X$ and $\div u = 0$, the elliptic equation \eqref{peq} 
has a solution $p \in \BMO(\R^2)$ given by \eqref{pdef}, which 
is unique up to an additive constant. 
\end{lemma}

The uniqueness claim in Lemma~\ref{plem} is easy to prove. If 
$\tilde p$ is another solution of \eqref{peq}, then $\tilde p - p$ 
is a harmonic function on $\R^2$, hence is identically constant
if we assume that $\tilde p \in \BMO(\R^2)$ (this can be 
seen as a slight generalization of Liouville's theorem). More 
generally, we could consider other solutions of \eqref{peq}, 
but since we want to solve the Navier-Stokes equations \eqref{2DNS}
in the space $X$ it is natural to assume that the pressure gradient
is bounded. So the most general admissible solution of \eqref{peq}
is $p + \alpha + \beta_1 x_1 + \beta_2 x_2$, where $p$ is given by 
\eqref{pdef} and $\alpha,\beta_1,\beta_2 \in \R$. The constant
$\alpha$ is irrelevant, but nonzero values of $\beta_1, \beta_2$
would correspond to driving the fluid by a pressure gradient 
(like, for instance, in the classical Poiseuille flow). In these
notes, we are interested in the intrinsic dynamics of the 
Navier-Stokes equations \eqref{2DNS} in the absence of exterior 
forcing, so we always use the canonical choice of the pressure
given by Lemma~\ref{plem}. 

\subsection{The Leray-Hopf projection}
\label{ss2.3}

With the canonical choice of the pressure \eqref{pdef}, the 
Navier-Stokes equations \eqref{2DNS} can be written in the 
equivalent form
\begin{equation}\label{NS2}
  \partial_t u + \P(u\cdot\nabla)u \,=\, \nu\Delta u~, 
  \qquad \div u \,=\, 0~,
\end{equation}
where the Leray-Hopf projection $\P$ is the matrix-valued operator
defined by 
\[
  (\P u)_j \,=\, \sum_{k=1}^2 \P_{jk}u_k~, \qquad \hbox{with}\quad 
  \P_{jk} = \delta_{jk} + R_j R_k~.
\]
Indeed, using Einstein's summation convention over repeated indices, 
we have from \eqref{pdef}\:
\[
  \partial_j p \,=\, \partial_j R_k R_\ell (u_k u_\ell) \,=\, 
  R_k R_j \partial_\ell (u_k u_\ell) \,=\, R_j R_k (u_\ell 
  \partial_\ell u_k)~,
\]
hence
\[
  \partial_j p + u_\ell\partial_\ell u_j \,=\, \Bigl(\delta_{jk} 
  + R_j R_k\Bigr) (u_\ell \partial_\ell u_k)~. 
\]
This shows that $\nabla p + (u\cdot \nabla)u = \P(u\cdot \nabla)u$. 
In the calculations above, we have used the commutations relations 
$R_1 R_2 = R_2 R_1$, $\partial_j R_k = R_k \partial_j$, $\partial_j 
R_\ell \,=\, R_j \partial_\ell$, as well as the incompressibility 
condition $\div u = 0$. Symbolically, we may also write
\begin{equation}\label{Pident}
  \P(u\cdot\nabla)u \,=\, \nabla\cdot \P (u\otimes u)~.
\end{equation}

It is clear that the Leray-Hopf projection $\P$ is a bounded 
linear operator on $L^p(\R^2)^2$ for $1 < p < \infty$, and 
from $L^\infty(\R^2)^2$ into $\BMO(\R^2)^2$. For later use,
we also mention that $\nabla S(t)\P$ defines a bounded operator
on $X = C_\bu(\R^2)^2$ for any $t > 0$, where $S(t)$ is the heat 
semigroup defined by \eqref{heatkernel}. 

\begin{lemma}\label{nablaSP}
There exists a constant $C_0 > 0$ such that
\begin{equation}\label{nablaSPest}
  \|\nabla S(t)\P f\|_{L^\infty} \,\le\, \frac{C_0}{\sqrt{t}}
  \,\|f\|_{L^\infty}~, \qquad t > 0~, 
\end{equation}
for all $f \in X$. 
\end{lemma}

\begin{proof}
For any $t > 0$ and any choice of $j,k,\ell \in \{1,2\}$, the 
operator $\partial_j S(t)\P_{k\ell}$ is the Fourier multiplier
with symbol
\[
  i\xi_j\Bigl(\delta_{k\ell} - \frac{\xi_k\xi_\ell}{|\xi|^2}\Bigr)
  \e^{-t|\xi|^2} \,=\,  i\xi_j \delta_{k\ell}\e^{-t|\xi|^2} - 
  i\xi_j \xi_k \xi_\ell \int_t^\infty \e^{-\tau|\xi|^2}\dd \tau~,
  \qquad \xi \in \R^2\setminus\{0\}~.
\]
We thus have the following identity
\begin{equation}\label{nablaSPid}
  \partial_j S(t)\P_{k\ell}f \,=\, \delta_{k\ell} \partial_j S(t)f 
  + \int_t^\infty \partial_j \partial_k \partial_\ell S(\tau)f \dd \tau~,
\end{equation}
which holds in particular for any $f \in C_\bu(\R^2)$. Both terms
in the right-hand side of \eqref{nablaSPid} belong to $C_\bu(\R^2)$
and can be easily estimated using \eqref{heatder}\:
\begin{align*}
  \|\partial_j S(t)f\|_{L^\infty} \,&\le\, \frac{C}{\sqrt{t}}\, 
  \|f\|_{L^\infty}~, \\
  \Bigl\|\int_t^\infty \partial_j\partial_k\partial_\ell S(\tau)f \dd \tau
  \Bigr\|_{L^\infty} \,&\le\, C \int_t^\infty \frac{1}{\tau^{3/2}}\,
  \|f\|_{L^\infty} \dd\tau \,\le\, \frac{C}{\sqrt{t}}\, \|f\|_{L^\infty}~, 
\end{align*}
This immediately yields estimates \eqref{nablaSPest}. 
\end{proof}

\subsection{Local existence of solutions}
\label{ss2.4}

Let $u_0 \in X$ be such that $\div u_0 = 0$. We consider the integral 
equation associated with the Navier-Stokes equations \eqref{NS2}\:
\begin{equation}\label{NSint}
  u(t) \,=\, S(\nu t)u_0 - \int_0^t \nabla\cdot S(\nu(t-s)) 
  \,\P (u(s) \otimes u(s))\dd s~, \qquad t > 0~,
\end{equation}
where $S(t)$ is the heat semigroup \eqref{heatkernel}, and 
the notation $\P(u\otimes u)$ is explained in \eqref{Pident}. 
Here and in the sequel, the map $x \mapsto u(x,t)$ is simply 
denoted by $u(t)$ instead of $u(\cdot,t)$. The goal of this 
section is to prove that the integral equation \eqref{NSint}
has a unique local solution that it continuous in time with 
values in the space $X$ defined by \eqref{Xdef}. Such a solution 
of \eqref{NSint} is usually called a {\em mild solution} of 
the Navier-Stokes equations \eqref{2DNS} in $X$. 

\begin{proposition}\label{locex}
Fix $\nu > 0$. For any $M > 0$, there exists a time $T = T(M,\nu) > 0$ 
such that, for all initial data $u_0 \in X$ with $\div u_0 = 0$ and 
$\|u_0\|_{L^\infty}\le M$, the integral equation \eqref{NSint} 
has a unique local solution $u \in C^0([0,T],X)$, which 
moreover satisfies
\begin{equation}\label{locexbd} 
  \sup_{0 \le t \le T}\|u(t)\|_{L^\infty} \,\le\, 2M~, \qquad
  \hbox{and}\quad 
  \sup_{0 < t \le T}(\nu t)^{1/2}\|\nabla u(t)\|_{L^\infty}
  \,\le\, C M~,
\end{equation}
where $C > 0$ is a universal constant. In addition, the solution
$u \in C^0([0,T],X)$ depends continuously on the initial data 
$u_0 \in X$. 
\end{proposition}

\begin{proof}
Take $T > 0$ small enough so that 
\begin{equation}\label{Tdef}
  \kappa \,:=\, 8 C_0 M \frac{T^{1/2}}{\nu^{1/2}} \,<\, 1~,
\end{equation}
where $C_0 > 0$ is as in Lemma~\ref{nablaSP}. We introduce the Banach 
space $Y = C^0([0,T],X)$ equipped with the norm 
\[
  \|u\|_Y \,=\, \sup_{0 \le t \le T} \|u(t)\|_{L^\infty}~.
\]
Using Lemma~\ref{nablaSP} it is not difficult to verify that, if $u
\in Y$, the integral in the right-hand side of \eqref{NSint} is well
defined and depends continuously on time in the topology of
$X$. Moreover, if $u_0 \in X$, the results of Section~\ref{ss2.1} show
that the map $t \mapsto S(\nu t)u_0$ also belongs to $Y$. Thus, given
any $u_0 \in X$ with $\div u_0 = 0$ and $\|u_0\|_{L^\infty}\le M$, we
can consider the map $F : Y \to Y$ defined by
\begin{equation}\label{Fdef}
  (Fu)(t) \,=\, S(\nu t)u_0 - \int_0^t \nabla\cdot S(\nu(t-s)) 
  \,\P (u(s) \otimes u(s))\dd s~, \qquad t \in [0,T]~.
\end{equation}
Denoting $B = \{u \in Y\,|\, \|u\|_Y \le 2M\} \subset Y$, we 
claim that

\smallskip\noindent{\bf i)} {\em $F$ maps $B$ into itself.} 
Indeed, if $u \in B$, we find using \eqref{nablaSPest} and
\eqref{Tdef}
\begin{align*}
  \|(Fu)(t)\|_{L^\infty} \,&\le\, \|u_0\|_{L^\infty} + \int_0^t 
  \frac{C_0}{\sqrt{\nu(t-s)}}\,\|u(s)\|_{L^\infty}^2 \dd s \\ 
  \,&\le\, M + \|u\|_Y^2 \int_0^t \frac{C_0}{\sqrt{\nu \tau}}
  \dd \tau \,\le\, M + 8 M^2 C_0 (T/\nu)^{1/2} \le (1+\kappa)M~, 
\end{align*}
for any $t \in [0,T]$. As $\kappa < 1$, we deduce that 
$\|Fu\|_Y < 2M$, hence $Fu \in B$. 

\smallskip\noindent{\bf ii)} {\em $F$ is a strict contraction in $B$.}
Indeed, if $u,v \in B$, then
\[
  (Fu)(t) - (Fv)(t) \,=\, \int_0^t \nabla\cdot S(\nu(t-s)) 
  \,\P \Bigl( (v(s) \otimes v(s)) - (u(s) \otimes u(s))\Bigr)\dd s~,
\]
hence decomposing $v\otimes v - u\otimes u = v\otimes (v-u) + (v-u)
\otimes u$ and proceeding as above, we find
\begin{align*}
  \|(Fu)(t) - (Fv)(t)\|_{L^\infty} \,&\le\, \int_0^t 
  \frac{C_0}{\sqrt{\nu(t-s)}}\,\Bigl(\|v(s)\|_{L^\infty} + 
  \|u(s)\|_{L^\infty}\Bigr)\,\|u(s) - v(s)\|_{L^\infty}\dd s \\
  \,&\le\, 4 M \|u-v\|_Y \int_0^t \frac{C_0}{\sqrt{\nu \tau}}\dd \tau 
  \,\le\, 8 M C_0 (T/\nu)^{1/2} \|u-v\|_Y~,
\end{align*}
for all $t \in [0,T]$. Thus $\|Fu - Fv\|_Y \le \kappa \|u-v\|_Y$ 
where $\kappa < 1$ is as in \eqref{Tdef}. 

\smallskip\noindent By the Banach fixed point theorem, the map $F : Y
\to Y$ has a unique fixed point $u$ in $B$, which satisfies by
construction the integral equation \eqref{NSint} as well as the first
bound in \eqref{locexbd}.  If $\tilde u \in Y$ is another solution of
\eqref{NSint}, then applying Gronwall's lemma to the integral equation
satisfied by the difference $\tilde u - u$ it is easy to verify that
$\tilde u = u$, see \cite{He} and Section~\ref{A3}. Thus the solution
$u$ of \eqref{NSint} constructed by the fixed point argument above is
unique not only in the ball $B$, but also in the whole space $Y$. A
similar argument shows that the solution $u \in Y$ is a locally
Lipschitz function of the initial data $u_0 \in X$. Finally, the
simplest way to prove the second bound in \eqref{locexbd} is to repeat
the existence proof using the smaller function space
\[
  Z \,=\, \Bigl\{u \in C^0([0,T],X)\,\Big|\, t^{1/2}\nabla u
  \in C^0_b((0,T],X^2)\Bigr\}~,
\]
equipped with the norm
\[
  \|u\|_Z \,=\, \sup_{0 \le t \le T} \|u(t)\|_{L^\infty} + 
  \sup_{0 < t \le T}(\nu t)^{1/2}\|\nabla u(t)\|_{L^\infty}~.
\]
Proceeding as above one obtains the existence of a local solution
$u \in Z$ of \eqref{NSint} for a slightly smaller value of 
$T$, which is determined by a condition of the form \eqref{Tdef}
where $C_0$ is replaced by a larger constant. 
\end{proof}

\begin{remark}\label{localtime}
If $u_0 \neq 0$, the local existence time $T$ given by the 
proof of Proposition~\ref{locex} satisfies
\begin{equation}\label{loctime}
  T \,=\, \frac{C_1\nu}{\|u_0\|_{L^\infty}^2}~,
\end{equation}
where $C_1 > 0$ is a universal constant. This implies in particular 
that, if we consider the maximal solution $u \in C^0([0,T_*),X)$ 
of \eqref{NSint} in $X$, then either $T_* = \infty$, which means
that the solution is global, or $\|u(t)\|_{L^\infty} \to \infty$ 
as $t \to T_*$. More precisely, we must have $\|u(t)\|_{L^\infty}^2
> C_1 \nu (T_* - t)^{-1}$ for all $t \in [0,T_*)$. Note also that
estimate \eqref{smoothT} follows from \eqref{locexbd} and 
\eqref{loctime}. 
\end{remark}

\begin{remark}\label{classicsol}
Using standard parabolic smoothing estimates, it is not difficult 
to show that, if $u \in C^0([0,T],X)$ is the mild solution of 
\eqref{2DNS} constructed in Proposition~\ref{locex}, then 
$u(x,t)$ is a smooth function for $(x,t) \in \R^2 \times (0,T]$
which satisfies the Navier-Stokes equations \eqref{2DNS} in the 
classical sense, with the pressure $p(x,t)$ given by 
\eqref{pdef}, see \cite{GIM,GMS}. 
\end{remark}

\subsection{Global existence and a priori estimates}
\label{ss2.5}

Take $u_0 \in X$ such that $\div u_0 = 0$, and let $u \in
C^0([0,T_*),X)$ be the maximal solution of \eqref{2DNS} with initial
data $u_0$, the existence of which follows from
Proposition~\ref{locex}. In view of Remark~\ref{localtime}, to prove
that this solution is global (namely, $T_* = \infty$), it is sufficient
to show that the norm $\|u(t)\|_{L^\infty}$ cannot blow up in finite
time. The easiest way to do that is to consider the vorticity
distribution $\omega = \curl u = \partial_1 u_2 - \partial_2 u_1$,
which satisfies the advection-diffusion equation
\begin{equation}\label{omeq2}
  \partial_t \omega + u\cdot\nabla \omega \,=\, \nu\Delta \omega~.
\end{equation}
We know from Proposition~\ref{locex} that $\|\omega(t)\|_{L^\infty}
\le 2 \|\nabla u(t)\|_{L^\infty} < \infty$ for any $t \in (0,T_*)$,
hence shifting the origin of time we can assume without loss of
generality that $\omega_0 = \omega(\cdot,0) \in L^\infty(\R^2)$. Now,
the parabolic maximum principle \cite{PW} asserts that
$\|\omega(t)\|_{L^\infty}$ is a {\em nonincreasing function of time},
which gives the a priori estimate
\begin{equation}\label{ombound}
  \|\omega(t)\|_{L^\infty} \,\le\, \|\omega_0\|_{L^\infty}~, \qquad 
  \hbox{for all } t \ge 0~.
\end{equation}
Unfortunately, the bound \eqref{ombound} does not provide any direct
control on $\|u(t)\|_{L^\infty}$, because of the low frequencies which
are due to the fact that we work in an unbounded domain. In 
Fourier space, the relation between $\hat u = \FF u$ and 
$\hat \omega = \FF \omega$ takes the simple form
\begin{equation}\label{FourierBS}
  \hat u(\xi) \,=\, \frac{-i\xi^\perp}{|\xi|^2}\,\hat\omega(\xi)~,
  \qquad \xi \in \R^2\setminus\{0\}~, 
\end{equation}
where $\xi^\perp = (-\xi_2,\xi_1)$ if $\xi = (\xi_1,\xi_2) \in \R^2$. 
This shows that the first-order derivatives of the 
velocity field $u$ satisfy
$$
  \partial_1 u_1 = -\partial_2 u_2 \,=\, R_1 R_2 \,\omega~, 
  \qquad \partial_1 u_2 \,=\, -R_1^2 \,\omega~, \qquad
  \partial_2 u_1 \,=\, R_2^2 \,\omega~,
$$
where $R_1, R_2$ are the Riesz transforms \eqref{Rieszdef}.
In particular, we deduce the a priori estimate
\begin{equation}\label{nablaubound}
  \|\nabla u(t)\|_{\BMO} \,\le\,  C\|\omega(t)\|_{L^\infty} 
 \,\le\, C \|\omega_0\|_{L^\infty}~, \qquad 
  \hbox{for all } t \ge 0~.
\end{equation}
We refer to Section~\ref{A1} for a more detailed discussion 
of the Biot-Savart law in $\R^2$. 

To go further we observe that, since $\div u = 0$, we have the 
identity
\begin{equation}\label{nlinid}
  (u\cdot\nabla)u \,=\, \half \nabla |u|^2 + u^\perp \omega~,
\end{equation}
where $u^\perp = (-u_2,u_1)$ if $u = (u_1,u_2)$. We can thus write 
the Navier-Stokes equations \eqref{2DNS} in the equivalent form
\begin{equation}\label{2DNS2}
  \partial_t u +  u^\perp \omega \,=\, \nu\Delta u - \nabla q~, 
  \qquad \div u \,=\, 0~,
\end{equation}
where $q = p + \half |u|^2$. Applying the Leray-Hopf projection, 
we obtain the analog of \eqref{NS2}
\begin{equation}\label{NS3}
  \partial_t u + \P(u^\perp\omega) \,=\, \nu\Delta u~, 
  \qquad \div u \,=\, 0~.
\end{equation}
Since $\omega$ is under control, the nonlinear term $\P(u^\perp\omega)$
can be considered as a linear expression in the velocity field $u$, 
and this strongly suggests that the solutions of \eqref{NS3} 
should not grow faster than $\exp(C\|\omega_0\|_{L^\infty}t)$ as 
$t \to \infty$. The problem with this naive argument is that 
we cannot control $\|\P(u^\perp\omega)\|_{L^\infty}$ in terms 
of $\|u\|_{L^\infty}\|\omega\|_{L^\infty}$, because the Leray-Hopf
projection $\P$ is not continuous on $L^\infty(\R^2)^2$. This 
difficulty was solved in an elegant way by O.~Sawada and Y.~Taniuchi, 
who obtained the following result. 

\begin{proposition}\label{globex} {\bf \cite{ST}}
Assume that $u_0 \in X$, $\div u_0 = 0$, and $\omega_0 = \curl u_0 
\in L^\infty(\R^2)$. Then the Navier-Stokes equations \eqref{2DNS} 
have a unique global (mild) solution $u \in C^0([0,\infty),X)$ with 
initial data $u_0$. Moreover, we have the estimate
\begin{equation}\label{globexbd}
  \|u(t)\|_{L^\infty} \,\le\, C\|u_0\|_{L^\infty} 
  \exp\Bigl(C\|\omega\|_{L^\infty} t\Bigr)~, \qquad t \ge 0~,
\end{equation}
for some universal constant $C > 0$. 
\end{proposition}

The proof of Proposition~\ref{globex} relies on a clever 
Fourier-splitting argument which we now describe. Let 
$\hat\chi : \R^2 \to \R$ be a smooth function such that
\[
  \hat\chi(\xi) \,=\, \begin{cases} 1 & \hbox{if} \quad|\xi| \le 1\,, \\
  0 & \hbox{if} \quad|\xi| \ge 2\,. \end{cases}
\]
We further assume that $\chi$ is radially symmetric and nonincreasing
along rays. Let $\chi = \FF^{-1}\hat\chi$ be the inverse Fourier 
transform of $\hat \chi$, so that $\chi \in \SS(\R^2)$. Given any
$\delta > 0$, we denote by $Q_\delta$ the Fourier multiplier 
with symbol $\hat\chi(\xi/\delta)$\: 
\begin{equation}\label{Qdef}
  (\widehat{Q_\delta f})(\xi) \,=\, \hat\chi(\xi/\delta) \hat 
  f(\xi)~, \qquad \xi \in \R^2~.
\end{equation}
It is clear that $Q_\delta$ is a bounded linear operator on 
$\SS'(\R^2)$. 

\begin{lemma}\label{Qest}
There exists a constant $C_2 > 0$ such that the following 
bounds hold for any $\delta > 0$. 

\smallskip\noindent\quad
{\bf 1.} $\|Q_\delta f\|_{L^\infty} \le C_2 \|f\|_{L^\infty}$, for any 
$f \in C_\bu(\R^2)$; 

\smallskip\noindent\quad
{\bf 2.} $\|Q_\delta \nabla \,\P f\|_{L^\infty} \le C_2\delta 
\|f\|_{L^\infty}$, for any $f \in C_\bu(\R^2)^2$; 

\smallskip\noindent\quad
{\bf 3.} $\|(\1 - Q_\delta) u\|_{L^\infty} \le C_2\delta^{-1} 
\|\omega\|_{L^\infty}$, for any $u \in X$ with $\div u = 0$ 
and $\curl u = \omega$.  
\end{lemma}

\begin{proof}
The first estimate follows immediately from Young's inequality (see
Section~\ref{A3}), because $Q_\delta$ is the convolution operator with
the integrable function $x \mapsto \delta^2 \chi(\delta x)$, the $L^1$
norm of which does not depend on $\delta$. To prove the second
estimate we have to show that, for any $j,k,\ell \in \{1,2\}$, the
Fourier multiplier $M$ with symbol
\[
  m(\xi) \,=\, \frac{i\xi_j \xi_k \xi_\ell}{|\xi|^2} 
  \,\hat\chi(\xi/\delta)~, \qquad \xi \in \R^2\setminus\{0\}~,
\]
is continuous on $L^\infty(\R^2)$ with operator norm bounded 
by $C \delta$. We observe that
\[
  m(\xi) \,=\, \delta \hat \psi(\xi/\delta)~, \qquad \hbox{where}
  \quad \hat \psi(\xi) \,=\, \frac{i\xi_j \xi_k \xi_\ell}{|\xi|^2} 
  \,\hat\chi(\xi)~.
\]
It follows that $Mf = \psi_\delta * f$, where $\psi_\delta(x) =
\delta^3 \psi(\delta x)$ and $\psi = \FF^{-1}\hat \psi$. It 
is clear that $\psi \in C^\infty(\R^2)$ (because $\hat\psi$ 
has compact support), and from the explicit formula
\[
  \psi(x) \,=\, \frac{1}{2\pi}\,\partial_j \partial_k \partial_\ell
  \int_{\R^2} \log(|x-y|) \chi(y)\dd y~, \qquad x \in \R^2~,
\]
it is straightforward to verify that $|\psi(x)| \le C|x|^{-3}$ 
for $|x| \ge 1$. Thus $\psi \in L^1(\R^2)$, and using Young's
inequality we conclude that $\|Mf\|_{L^\infty} \le \|\psi_\delta\|_{L^1} 
\|f\|_{L^\infty} = \delta \|\psi\|_{L^1} \|f\|_{L^\infty}$, which 
is the desired result. 

Finally, to prove the third estimate in Lemma~\ref{Qest}, 
we use formula~\eqref{FourierBS} to derive the relation
\[
  \hat u(\xi) - \widehat{Q_\delta u}(\xi) \,=\, \Bigl(1 - \hat\chi
  (\xi/\delta)\Bigr) \frac{-i\xi^\perp}{|\xi|^2}\,\hat\omega(\xi) 
  \,=\, \frac{1}{\delta}\,\hat\phi(\xi/\delta)\hat \omega(\xi)~, 
\]
where 
\begin{equation}\label{hatphidef}
  \hat\phi(\xi) \,=\, \Bigl(1 - \hat\chi(\xi)\Bigr)
  \frac{-i\xi^\perp}{|\xi|^2}~, \qquad \xi \in \R^2\setminus\{0\}~. 
\end{equation}
As before, if $\phi = \FF^{-1}\hat \phi$, this implies that 
$\|(\1 -Q_\delta) u\|_{L^\infty} \le \delta^{-1} \|\phi\|_{L^1} 
\|\omega\|_{L^\infty}$, so we only need to verify that $\phi \in
L^1(\R^2)$. Since $\hat\chi$ has compact support in $\R^2$, it
follows from \eqref{hatphidef} that
\[
  \phi(x) \,=\, \frac{1}{2\pi} \frac{x^\perp}{|x|^2} + \Phi(x)~, 
  \qquad x \in \R^2\setminus\{0\}~, 
\]
where $\Phi : \R^2 \to \R$ is smooth, hence $\phi$ is integrable on
any bounded neighborhood of the origin. On the other hand, if we apply
the Laplacian $\Delta_\xi$ to both sides of \eqref{hatphidef}, the
resulting expression belongs to $L^2(\R^2,\D\xi)$. This shows that
$|x|^2\phi \in L^2(\R^2,\D x)$, hence $\phi$ is integrable on the
complement of any neighborhood of the origin. Thus altogether $\phi
\in L^1(\R^2)$, which is the desired result.
\end{proof}

\noindent{\it Proof of Proposition~\ref{globex}.}
Let $u \in C^0([0,T_*),X)$ be the maximal solution of \eqref{2DNS}
with initial data $u_0$. Without loss of generality, we assume
that $u_0 \not\equiv 0$, and we fix $t \in (0,T_*)$. The idea is to
control the low frequencies $|\xi| \le 2\delta$ in the solution $u(t)$
using the integral equation \eqref{NSint}, and the high frequencies
$|\xi| \ge \delta$ using the third estimate in Lemma~\ref{Qest} 
together with the a priori bound on the vorticity. The threshold 
frequency $\delta$ will depend on time and on the solution itself. 

Given any $\delta > 0$, we apply the Fourier multiplier $Q_\delta$
defined in \eqref{Qdef} to the integral equation \eqref{NSint} 
and obtain
\[
  Q_\delta u(t) \,=\, S(\nu t)Q_\delta u_0 - \int_0^t S(\nu(t-s)) 
  Q_\delta \nabla\cdot \P (u(s)\otimes u(s))\dd s~,
\]
where we have used the fact that $Q_\delta$ commutes with the heat 
semigroup $S(t)$. Using the first two estimates in Lemma~\ref{Qest}, 
we thus find
\[
  \|Q_\delta u(t)\|_{L^\infty} \,\le\,  C_2 \|u_0\|_{L^\infty} + 
  C_2 \delta \int_0^t \|u(s)\|_{L^\infty}^2\dd s~.
\]
On the other hand, the third estimate in Lemma~\ref{Qest} 
implies that
\[
  \|(\1 - Q_\delta)u(t)\|_{L^\infty} \,\le\, C_2 \delta^{-1} 
  \|\omega(t)\|_{L^\infty} \,\le\, C_2 \delta^{-1} \|\omega_0\|_{L^\infty}~. 
\]
This bound shows how the high frequencies in the velocity field
$u(t)$ can be controlled in terms of the vorticity. Combining both 
results, we find
\[
  \|u(t)\|_{L^\infty} \,\le\, C_2 \|u_0\|_{L^\infty} + C_2 \delta 
  \int_0^t \|u(s)\|_{L^\infty}^2\dd s + C_2 \delta^{-1} 
  \|\omega_0\|_{L^\infty}~.
\]
If we now choose
\[
  \delta \,=\, \|\omega_0\|_{L^\infty}^{1/2} \,\left(\int_0^t \|u(s)
  \|_{L^\infty}^2\dd s\right)^{-1/2}~,
\]
we obtain the bound
\begin{equation}\label{ufirstbound}
  \|u(t)\|_{L^\infty} \,\le\, C_2 \|u_0\|_{L^\infty} + 2C_2 
  \|\omega_0\|_{L^\infty}^{1/2} \,\left(\int_0^t \|u(s)
  \|_{L^\infty}^2\dd s\right)^{1/2}~, 
\end{equation}
which holds for any $t \in (0,T_*)$. Finally, squaring 
both sides of \eqref{ufirstbound} and applying Gronwall's
lemma (see Section~\ref{A3}), we arrive at the inequality
\begin{equation}\label{usecondbound}
  \|u(t)\|_{L^\infty}^2 \,\le\, 2C_2^2 \|u_0\|_{L^\infty}^2
  \,\exp\Bigl(8C_2^2  \|\omega_0\|_{L^\infty}t\Bigr)~, \qquad
  t \in (0,T_*)~,
\end{equation}
which shows that the norm $\|u(t)\|_{L^\infty}$ cannot blow 
up in finite time. Thus $T_* = \infty$, and estimate 
\eqref{usecondbound} holds for all $t > 0$. \QED

\begin{remark}
Theorem~\ref{thm1} is an immediate consequence of 
Propositions~\ref{locex} and \ref{globex}. 
\end{remark}

\section{Uniformly local energy estimates}
\label{s3}

In the study of nonlinear partial differential equations on unbounded
spatial domains, if one considers solutions that do not decay to zero
at infinity, it is not always convenient to use function spaces based
on the uniform norm $\|\cdot\|_\infty$, because those spaces may not
take into account some essential properties of the system, such as
locally conserved or locally dissipated quantities. From this point of
view, the larger family of {\em uniformly local Lebesgue spaces}
offers an interesting compromise between simplicity and flexibility. 
In the analysis of evolution PDE's, uniformly local spaces were 
introduced by T.~Kato in 1975 \cite{Ka}, and subsequently used by 
many authors, see \cite{ARCD,EZ,Fe,GS0,GV,MT,MS,Ze} for a few 
examples.

The following two sections are largely independent of the rest of
these notes. The first one provides the definition and the main
properties of the uniformly local Lebesgue spaces, including various
characterizations of their norm. In the subsequent section, we
consider the simple example of the linear heat equation on $\R^d$ and
show how the solutions can be controlled using uniformly local energy
estimates. These preliminaries are useful to understand the general
philosophy of our approach, but are not necessary to follow the 
proof of our main results. The impatient reader should jump 
directly to Section~\ref{ss3.3}. 

\subsection{Uniformly local Lebesgue spaces}
\label{ss3.1}

Let $d \in \N^*$ and $1 \le p < \infty$. We introduce the space 
$\LL^p_\ul(\R^d)$ defined by
\begin{equation}\label{LLpuldef}
  \LL^p_\ul(\R^d) \,=\, \Bigl\{f \in L^p_\loc(\R^d)\,\Big|\,
  \|f\|_{L^p_\ul} < \infty\Bigr\}~,
\end{equation}
where
\begin{equation}\label{ulnorm}
  \|f\|_{L^p_\ul} \,=\, \sup_{x \in \R^d}\left(\int_{|y-x|\le 1}
  |f(y)|^p \dd y\right)^{1/p}~.
\end{equation}
In other words, a function $f$ belongs to $\LL^p_\ul(\R^d)$ if and
only if $f \in L^p(B_x)$ for any $x \in \R^d$, where $B_x \subset
\R^d$ denotes the ball of unit radius centered at $x$, and if moreover
the norm $\|f\|_{L^p(B_x)}$ is uniformly bounded for all $x \in
\R^2$. Roughly speaking, a function $f \in \LL^p_\ul(\R^d)$ is locally
in $L^p$ but behaves at large scales like a bounded function.

The {\em uniformly local $L^p$ space} $L^p_\ul(\R^d)$ is the subspace 
of $\LL^p_\ul(\R^d)$ defined by
\begin{equation}\label{Lpuldef}
  L^p_\ul(\R^d) \,=\, \Bigl\{f \in \LL^p_\ul(\R^d) \,\Big|\,
  \|\tau_y f - f\|_{L^p_\ul} \xrightarrow[y \to 0]{} 0\Bigr\}~,
\end{equation}
where $\tau_y$ denotes the translation operator\: $(\tau_y f)(x) = 
f(x-y)$ for $x,y \in \R^d$. The following properties are 
well known \cite{ARCD}\:

\begin{enumerate}

\item The space $\LL^p_\ul(\R^d)$ equipped with the norm \eqref{ulnorm}
is a Banach space, which contains $L^p_\ul(\R^d)$ as a closed subspace. 
In fact $L^p_\ul(\R^d)$ is the closure of $C_\bu(\R^d)$ in 
$\LL^p_\ul(\R^d)$, so that $C_\bu(\R^d)$ and even $C^\infty_\bu(\R^d)$ 
are dense in $L^p_\ul(\R^d)$. 

\item If $p = \infty$ the norm \eqref{ulnorm} should be understood 
as the uniform norm over $\R^d$. We thus have $\LL^\infty_\ul(\R^d)
= L^\infty(\R^d)$ and the definition \eqref{Lpuldef} shows that
$L^\infty_\ul(\R^d) = C_\bu(\R^d)$.  

\item For any $p \in [1,\infty]$ one has $L^p_\ul(\R^d) \neq 
\LL^p_\ul(\R^d)$. For instance, if $f(x) = \sin(|x|^2)$, it is 
easy to verify that $f \in \LL^p_\ul(\R^d)\setminus L^p_\ul(\R^d)$. 
Such a function cannot be approximated by uniformly continuous
functions in the topology defined by the norm \eqref{ulnorm}. 

\item As a Banach space, $L^p_\ul(\R^d)$ is neither reflexive nor
separable. 

\item If $1 \le p \le q \le \infty$ one has the embeddings
$C_\bu(\R^d) \hookrightarrow   L^q_\ul(\R^d)\hookrightarrow  
L^p_\ul(\R^d) \hookrightarrow  L^1_\ul(\R^d)$.  

\end{enumerate}

\noindent
Uniformly local Sobolev spaces can be constructed in a similar 
way. For instance one can define $W^{1,p}_\ul(\R^d)$ as the space 
of all $f \in L^p_\ul(\R^d)$ such that the distributional derivatives 
$\partial_i f$ belong to $L^p_\ul(\R^d)$ for $i = 1,\dots,d$. 

Uniformly local Lebesgue spaces provide a convenient framework 
for solving evolution PDE's. As a simple example, we consider the 
linear heat equation $\partial_t u = \Delta u$ in $\R^d$. 
The solution with initial data $u_0 \in L^p_\ul(\R^d)$ is 
$u(t) = S(t)u_0$, where $S(0) = \1$ and 
\begin{equation}\label{heatkernel2}
  (S(t)u_0)(x) \,=\, \frac{1}{(4\pi t)^{d/2}} \int_{\R^d} 
  \e^{-\frac{|x-y|^2}{4t}} u_0(y) \dd y~, \qquad x \in \R^d~, \quad t > 0~.
\end{equation}
The following result is not difficult to establish \cite{ARCD}\:

\begin{proposition}\label{heatul}
The family $\{S(t)\}_{t\ge 0}$ given by \eqref{heatkernel2} defines
a strongly continuous semigroup on $L^p_\ul(\R^d)$ for 
$1 \le p \le \infty$. Moreover, if $1 \le p \le q \le 
\infty$ one has the estimate
\begin{equation}\label{heatulest}
  \|S(t)u_0\|_{L^q_\ul} \,\le\, C \Bigl(1 + t^{-\frac{d}{2}(\frac1p-\frac1q)}
  \Bigr)\|u_0\|_{L^p_\ul}~, \qquad t > 0~.
\end{equation} 
\end{proposition}

In Proposition~\ref{heatul}, it is important to use $L^p_\ul(\R^d)$
instead of the larger space $\LL^p_\ul(\R^d)$, because if $u_0 \in
\LL^p_\ul(\R^d) \setminus L^p_\ul(\R^d)$ the solution $u(t) = S(t)u_0$
is not right continuous at $t = 0$.  For short times ($t \le 1$), the
bound \eqref{heatulest} reduces to the usual $L^p-L^q$ estimate for
the heat semigroup in $\R^d$, whereas for large times ($t \ge 1$) we
recover the $L^\infty-L^\infty$ estimate. This is not surprising if
one remembers that elements of $L^p_\ul(\R^d)$ behave locally like $L^p$
functions, but look like bounded functions when considered at a
sufficiently large scale. It is also possible to obtain smoothing
estimates for the heat semigroup in uniformly local Lebesgue 
spaces. For instance, if $1 \le p \le q \le \infty$, we have
\begin{equation}\label{heatulder}
  \|\nabla S(t)u_0\|_{L^q_\ul} \,\le\, C t^{-1/2}\Bigl(1 + 
  t^{-\frac{d}{2}(\frac1p-\frac1q)}\Bigr)\|u_0\|_{L^p_\ul}~, \qquad t > 0~.
\end{equation}

In the applications to partial differential equations, it is often
convenient to use slightly different norms on $L^p_\ul(\R^d)$
which turn out to be equivalent to \eqref{ulnorm}. Let $\rho : 
\R^d \to \R_+$ be a measurable function with the following 
two properties\:

\smallskip\noindent\quad{\bf a)} $\rho$ is positive on a set of 
nonzero measure; 

\smallskip\noindent\quad{\bf b)} $\tilde\rho \in L^1(\R^d)$, where
$\tilde \rho(x) = \sup\{\rho(y)\,|\, |y-x| \le 1\}$. 

\smallskip\noindent
The second assumption implies that $\rho \in L^1(\R^d) \cap 
L^\infty(\R^d)$. Indeed, we clearly have $\rho \le \tilde \rho$, 
hence $\|\rho\|_{L^1} \le \|\tilde \rho\|_{L^1}$. Moreover, 
if $B \subset \R^d$ is any ball of unit diameter, the 
definition of $\tilde \rho$ implies that $\rho(x) \le \tilde 
\rho(y)$ for all $x,y \in B$, hence
\begin{equation}\label{rhocomp}
  \sup_{x \in B}\rho(x) \,\le\, \inf_{y \in B}\tilde\rho(y)
  \,\le\, \frac{1}{|B|}\int_B \tilde\rho(y)\dd y~. 
\end{equation}
This shows that $\|\rho\|_{L^\infty} \le |B|^{-1}\|\tilde \rho\|_{L^1}$.
Assumption a) ensures that $\rho$ is not zero almost everywhere, 
so that $\int_{\R^d} \rho\dd x > 0$. 

The following result provides a plethora of equivalent norms 
on $L^p_\ul(\R^d)$. 

\begin{proposition}\label{rholem}
If $\rho : \R^d \to \R_+$ satisfies assumptions a) and b) above, 
then for $1 \le p < \infty$ the quantity 
\begin{equation}\label{Lprhodef}
  \|f\|_{p,\rho} \,=\, \sup_{x \in \R^d} \left(\int_{\R^d} 
  \rho(x-y)|f(y)|^p\dd y\right)^{1/p}
\end{equation}
is equivalent to the norm $\|\cdot\|_{L^p_\ul}$ on $L^p_\ul(\R^d)$. 
\end{proposition}

\begin{remark}\label{obvious}
Obviously, if $\rho$ is the characteristic function of the unit ball 
$B_0 \subset \R^d$, the definition \eqref{Lprhodef} reduces to 
\eqref{Lpuldef}. So \eqref{Lprhodef} is clearly a generalization 
of \eqref{Lpuldef}. 
\end{remark}

\begin{proof}
For any $z \in \R^d$, we denote by $Q_z \subset \R^d$ the cube of 
unit diameter centered at $z$, the edges of which are parallel 
to the coordinate axes. The Lebesgue measure of $Q_z$ is 
$\lambda^d$, where $\lambda = d^{-1/2}$. Given any $f \in 
\LL^p_\ul(\R^d)$ and any $x \in \R^d$, we estimate
\begin{equation}\label{rhofirst}
  \int_{\R^d} \rho(x-y)|f(y)|^p\dd y \,=\, \sum_{k \in \lambda\Z^d} 
  \int_{Q_k} \rho(x-y)|f(y)|^p\dd y \,\le\, \sum_{k \in \lambda\Z^d} 
  r(x-k) \int_{Q_k}|f(y)|^p\dd y~, 
\end{equation}
where $r(x-k) = \sup\{\rho(x-y)\,|\, y \in Q_k\} = \sup\{\rho(z)\,|\, 
z \in Q_{x-k}\}$. We observe that
\[
  \int_{Q_k}|f(y)|^p\dd y \,\le\, \int_{B_k}|f(y)|^p\dd y 
  \,\le\, \|f\|_{L^p_\ul}^p~, 
\]
because $Q_k \subset B_k$ (the ball of unit radius centered 
at $k$). On the other hand, we have as in \eqref{rhocomp}\: 
\[
  r(x-k) \,=\, \sup_{z \in Q_{x-k}}\rho(z) \,\le\, 
  \inf_{y \in Q_{x-k}}\tilde \rho(y) \,\le\, \frac{1}{\lambda^d}
  \int_{Q_{x-k}} \tilde \rho(y)\dd y~,
\] 
hence
\[
  \sum_{k \in \lambda\Z^d}r(x-k) \,\le\, \frac{1}{\lambda^d}
  \sum_{k \in \lambda\Z^d}\int_{Q_{x-k}} \tilde \rho(y)\dd y 
  \,=\, d^{d/2}\|\tilde \rho\|_{L^1}~. 
\]
Taking the supremum over $x \in \R^d$ in \eqref{rhofirst}, we 
conclude that $\|f\|_{p,\rho}^p \le d^{d/2}\|\tilde \rho\|_{L^1} 
\|f\|_{L^p_\ul}^p$. 

Conversely, as $\rho$ is positive on a set of nonzero measure, 
we can assume without loss of generality that $\int_{B_0} \rho 
\dd x = \epsilon > 0$, where $B_0 \subset \R^d$ is the unit 
ball centered at the origin. If $\|f\|_{L^p_\ul} > 0$, we choose 
$x \in \R^d$ such that
\[
  \int_{|y-x|\le 1}|f(y)|^d\dd y \,\ge\, \frac{1}{2}
  \|f\|_{L^p_\ul}^p~.
\]
We then have
\[
  \int_{|z|\le 2} \biggl\{\int_{\R^d} \rho(x{-}y{-}z)|f(y)|^p\dd y
  \biggr\}\dd z \,=\, \int_{\R^d} \biggl\{\int_{|z|\le 2} \rho(x{-}y{-}z)
  \dd z\biggr\}|f(y)|^p\dd y\,\ge\, \frac{\epsilon}{2}
  \|f\|_{L^p_\ul}^p~,
\]
because by assumption $\int_{|z|\le 2} \rho(x{-}y{-}z)\dd z \ge 
\epsilon$ whenever $|y-x| \le 1$. Thus there exists $z \in 
\R^d$ with $|z| \le 2$ such that
\[
  \|f\|_{p,\rho}^p \,\ge\, \int_{\R^d} \rho(x{-}y{-}z)|f(y)|^p\dd y
  \,\ge\, \frac{\epsilon}{2}\meas\{z \in \R^d\,|\, |z|\le 2\}^{-1}
  \|f\|_{L^p_\ul}^p~.
\]
This proves the desired equivalence. 
\end{proof}

\begin{remark}\label{charac}
Proposition~\ref{rholem} provides sufficient conditions on the weight 
function $\rho$ so that \eqref{Lprhodef} is equivalent 
to \eqref{Lpuldef}. These conditions are weaker than what 
can be found in the existing literature (see e.g. 
\cite[Definition~4.1]{ARCD}), but it is not clear that 
assumptions a) and b) are optimal. It is easy to verify that
any weight $\rho$ for which \eqref{Lprhodef} is equivalent to 
\eqref{Lpuldef} should satisfy $\rho \in L^1(\R^d) \cap 
L^\infty(\R^d)$ and $\int \rho\dd x > 0$, but these properties
alone are not sufficient, as can be seen from the following 
example. Assume that $d = 1$ and take
\[
  \rho \,=\, \sum_{k=1}^\infty \frac{1}{\sqrt{k}}\,\1_{[-k-k^{-1},-k]}~,
  \qquad f \,=\, \sum_{k=1}^\infty k\,\1_{[k,k+k^{-1}]}~,  
\] 
where $\1_I$ denotes the characteristic function of an 
interval $I \subset \R$. Then $\rho \in L^1(\R) \cap 
L^\infty(\R)$, and using definition \eqref{Lpuldef} we 
see that $f \in \LL^1_\ul(\R)$. But
\[
  \int_\R \rho(-y) f(y) \dd x \,=\, \sum_{k=1}^\infty \frac{1}{\sqrt{k}}
  \,=\, +\infty~,
\]
so that $\|f\|_{1,\rho} = +\infty$. 
\end{remark}

\subsection{Uniformly local energy estimates for the heat  equation}
\label{ss3.2}

In this section we show on a simple example how uniformly local energy
estimates can be used to obtain information on the solutions of
partial differential equations on unbounded domains. We concentrate on
the linear heat equation on $\R^d$, with nondecaying initial data
$u_0$. In that particular example, the solution can be written in
explicit form, but we shall not use the heat kernel
\eqref{heatkernel2} because we want to develop robust methods that can
be applied to more complicated situations, such as the two-dimensional
Navier-Stokes equations which will considered later.

Let $u_0 \in L^2_\ul(\R^d)$, and let $u(x,t)$ be the solution of 
the heat equation
\begin{equation}\label{heateq2}
  \partial_t u(x,t) \,=\, \Delta u(x,t)~, \qquad x \in \R^d~, \quad
  t > 0~,
\end{equation}
with initial data $u(\cdot,0) = u_0$. We know from 
Proposition~\ref{heatul} that $u \in C^0(\R_+, L^2_\ul(\R^d))$, and
our goal is to derive accurate bounds on $u$ using localized 
energy estimates. Let $\rho : \R^d \to (0,+\infty)$ be a 
Lipschitz continuous function satisfying the assumptions 
of Proposition~\ref{rholem} and such that $|\nabla\rho(x)| \le \rho(x)$
for almost every $x \in \R^d$. Typical examples are 
\begin{equation}\label{rhoex}
  \rho(x) \,=\, e^{-|x|}~, \qquad \hbox{or}\qquad 
  \rho(x) \,=\, \frac{1}{(m+|x|)^m} \quad\hbox{where }m > d~.
\end{equation}
Note that $\rho$ cannot decay to zero faster than an exponential
as $|x| \to \infty$, because of the assumption $|\nabla\rho| \le \rho$. 
For any $R > 0$, we also define $\rhoR(x) = \rho(x/R)$. 

Since the solution $u$ of \eqref{heateq2} is smooth and bounded 
for $t > 0$, we can compute
\begin{align*}
  \frac{1}{2}\frac{\D}{\D t}\int_{\R^d} \rhoR u^2 \dd x \,&=\,
  \int_{\R^d} \rhoR u u_t \dd x \,=\, \int_{\R^d} \rhoR u 
  \Delta u \dd x \\ \,&=\, -\int_{\R^d} \rhoR |\nabla u|^2 \dd x 
  -\int_{\R^d} (\nabla\rhoR\cdot\nabla u)u \dd x \\
  \,&\le\, -\int_{\R^d} \rhoR |\nabla u|^2 \dd x + \frac{1}{R}
  \int_{\R^d} \rhoR|\nabla u||u| \dd x \\ \,&\le\, -\frac12\int_{\R^d} 
  \rhoR |\nabla u|^2 \dd x + \frac{1}{2R^2}\int_{\R^d} \rhoR u^2 
  \dd x~.
\end{align*}
Using Gronwall's lemma (see Section~\ref{A3}), we deduce that
\begin{equation}\label{hunif1}
  \int_{\R^d} \rhoR(x) u(x,t)^2 \dd x + \int_0^t \!\int_{\R^d} 
  \rhoR(x) |\nabla u(x,s)|^2 \dd x\dd s \,\le\, 
  \biggl(\int_{\R^d} \rhoR(x) u_0(x)^2 \dd x\biggr)\e^{t/R^2}~, 
\end{equation}
for all $t > 0$. This estimate looks rather pessimistic, because
it predicts an exponential growth of the solution as $t \to \infty$, 
but one should keep in mind that $\int\rhoR u_0^2 \dd x < \infty$
is the only assumption on the initial data that was really used in 
the derivation of \eqref{hunif1}. If $\rho(x) = \e^{-|x|}$, this means 
that $u_0$ is allowed to grow exponentially as $|x| \to \infty$, in 
which case the solution of \eqref{heateq2} indeed grows exponentially 
in time. 

To improve \eqref{hunif1} for large times, we must use the assumption 
that $u_0 \in L^2_\ul(\R^d)$. If $R \ge 1$, an easy calculation 
shows that
\begin{equation}\label{hunif2}
  \int_{\R^d} \rhoR(x) u_0(x)^2 \dd x \,\le\, C_d R^d \|u_0\|_{L^2_\ul}^2~, 
\end{equation}
for some constant $C_d > 0$ depending on the dimension $d$. So if 
we take $R = R(t) = (1+t)^{1/2}$ we obtain from \eqref{hunif1}
and \eqref{hunif2}
\begin{equation}\label{hunif3}
  \int_{\R^d} \rhoRt(x) u(x,t)^2 \dd x + \int_0^t \!\int_{\R^d} 
  \rhoRt(x) |\nabla u(x,s)|^2 \dd x\dd s \,\le\, C
  \|u_0\|_{L^2_\ul}^2 (1+t)^{d/2}~, 
\end{equation}
for all $t > 0$. This estimate is clearly superior to \eqref{hunif1}, 
because the right-hand side grows only polynomially as $t \to \infty$. 
Another elementary but important observation is that the same estimate 
holds if we replace $\rhoR(x)$ with $\rhoR(x-y)$ for any fixed 
$y \in \R^d$. Taking the supremum over all translations and using 
the fact that $\rhoR(x) \ge c > 0$ whenever $|x| \le R$, we obtain 
\begin{equation}\label{hunif4}
  \sup_{x \in \R^d} \frac{1}{R(t)^d}\int_{|y-x| \le R(t)} 
  u(y,t)^2\dd y \,\le\, C \|u_0\|_{L^2_\ul}^2~, \qquad t \ge 0~,
\end{equation}
where $R(t) = (1+t)^{1/2}$ and $C > 0$ is independent of $t$. 
In particular, if $u_0 \equiv 1$, then $u(\cdot,t) \equiv 1$ 
for all positive times, and \eqref{hunif4} is sharp in that
particular case, as far as the time dependence is concerned.

Unfortunately, the approach developed so far does not allow
to bound the norm $\|u(t)\|_{L^2_\ul}$ in an optimal way. 
Indeed, the best we can deduce directly from \eqref{hunif3} is
\[
  \|u(t)\|_{L^2_\ul} \,\le\, C \|u_0\|_{L^2_\ul}(1+t)^{d/4}~,
  \qquad t \ge 0~,
\]
which is not sharp in view of Proposition~\ref{heatul}.
To improve that result, a possibility is to use uniformly 
local $L^p$ estimates for higher values of $p$. Indeed, 
if $p \in \N^*$ we have as before
\begin{align*}
  \frac{1}{2}\frac{\D}{\D t}\int_{\R^d} \rhoR u^{2p} \dd x \,&=\,
  p\int_{\R^d} \rhoR u^{2p-1}u_t \dd x \,=\,
  p\int_{\R^d} \rhoR u^{2p-1}\Delta u \dd x \\ 
  \,&=\, -p(2p{-}1)\int_{\R^d} \rhoR u^{2p-2}|\nabla u|^2 \dd x 
  -p\int_{\R^d} (\nabla\rhoR\cdot\nabla u)u^{2p-1} \dd x \\
  \,&=\, -\frac{2p{-1}}{p}\int_{\R^d} \rhoR |\nabla u^p|^2 \dd x 
  -\int_{\R^d} (\nabla\rhoR\cdot\nabla u^p)u^p \dd x \\
  \,&\le\, -\int_{\R^d} \rhoR |\nabla u^p|^2 \dd x 
  + \frac{1}{R}\int_{\R^d} \rhoR|\nabla u^p||u|^p \dd x 
  \\ \,&\le\, -\frac12\int_{\R^d} \rhoR |\nabla u^p|^2 \dd x + 
  \frac{1}{2R^2}\int_{\R^d} \rhoR u^{2p}\dd x~.
\end{align*}
Proceeding as above, we find if $u_0 \in L^{2p}_\ul(\R^d)$\:
\begin{align}\nonumber
  \int_{\R^d} \rhoRt(x) u(x,t)^{2p} \dd x \,\le\,  
  C \int_{\R^d} \rhoRt(x) u_0(x)^{2p} \dd x \,\le\, 
  C \|u_0\|_{L^{2p}_\ul}^{2p} (1+t)^{d/2}~, 
\end{align}
where $R(t) = (1+t)^{1/2}$, and this implies
\[
  \|u(t)\|_{L^{2p}_\ul} \,\le\, C^{\frac{1}{2p}} \|u_0\|_{L^{2p}_\ul}
  (1+t)^{\frac{d}{4p}}~, \qquad t \ge 0~.
\]
Here the constant $C$ does not depend on $p$. Now, if we assume 
that $u_0 \in C_\bu(\R^d)$, we can take the limit $p \to \infty$
in the above inequality, and we obtain the bound $\|u(t)\|_{L^\infty} 
\le \|u_0\|_{L^\infty}$ which is clearly optimal. 

\subsection{Uniformly local energy estimates for the $2D$ 
Navier-Stokes equations}
\label{ss3.3}

After these preliminaries, we return to the two-dimensional
Navier-Stokes equations in $\R^2$. We assume that the initial data 
$u_0 \in X$ satisfy $\div u_0 = 0 $ and $\omega_0 = \curl u_0 \in
L^\infty(\R^2)$. Let $u \in C^0([0,+\infty),X)$ be the unique 
mild solution of \eqref{2DNS} given by Proposition~\ref{globex}. 
Our goal is to control the velocity field $u(x,t)$ for large
times using uniformly local $L^2$ estimates. 

For technical reasons, related to the control of the pressure 
term in \eqref{2DNS}, it is convenient here to used compactly
supported weight functions, see \cite{Ze}. Let $\psi : 
\R^2 \to \R_+$ be a smooth function satisfying
\[
  \psi(x) \,=\, \begin{cases} 1 & \hbox{if} \quad|x| \le 1\,, \\
  0 & \hbox{if} \quad|x| \ge 2\,. \end{cases}
\]
We also assume that $\psi$ is radially symmetric and nonincreasing
along rays, and we define $\phi = \psi^2$. Then $\phi(x)$ is also
equal to $1$ if $|x| \le 1$ and to $0$ if $|x| \ge 2$. In addition, we
have the estimate
\begin{equation}\label{nablaphiest}
  |\nabla \phi(x)| \,\le\, C_3 \phi(x)^{1/2}~, \qquad x \in \R^2~,
\end{equation}
where $C_3 = 2\|\nabla \psi\|_{L^\infty}$. Note that a compactly
supported function $\phi$ cannot satisfy $|\nabla \phi| \le C \phi$,
unless $\phi \equiv 0$, and this is why we shall only use the weaker
property \eqref{nablaphiest}. Given $x_0 \in \R^2$ and $R > 0$, we
also consider the translated and rescaled localization function
$\phi_{R,x_0}$ defined by
\begin{equation}\label{phiRdef}
  \phi_{R,x_0}(x) \,=\, \phi\Bigl(\frac{x-x_0}{R}\Bigr)~, \qquad
  x \in \R^2.
\end{equation}
By construction we have $\phi_{R,x_0} = 1$ on $B_{x_0}^R$ and
$\phi_{R,x_0} = 0$ on the complement of $B_{x_0}^{2R}$, where
$B_{x_0}^R$ denotes the closed ball with radius $R$ centered 
at $x_0$\:
\[
  B_{x_0}^R \,=\, \Bigl\{x \in \R^2\,\Big|\, |x-x_0| \le R\Bigr\}~.
\]

Our starting point is the following localized energy estimate\:

\begin{lemma}\label{loclem}
For any $x_0 \in \R^2$ and any $R > 0$, the solution of \eqref{2DNS}
satisfies
\begin{align}\nonumber
  \frac12 \frac{\D}{\D t}\int_{\R^2} \phi_{R,x_0} |u|^2 \dd x &+ 
  \nu\int_{\R^2} |\nabla(\phi_{R,x_0}^{1/2}u)|^2 \dd x \\ \label{locest}
  \,&=\, \nu\int_{\R^2} |\nabla \phi_{R,x_0}^{1/2}|^2 |u|^2 \dd x +
  \int_{\R^2} q (u\cdot \nabla\phi_{R,x_0})\dd x~,   
\end{align}
where $q = p + \half |u|^2$.
\end{lemma}

\begin{proof}
For simplicity we write $\phi$ instead of $\phi_{R,x_0}$, and we 
denote $\psi = \phi^{1/2}$. Using the equivalent form \eqref{2DNS2}
of the Navier-Stokes equations, we easily obtain
\begin{equation}\label{locest1}
 \frac12 \frac{\D}{\D t}\int_{\R^2} \phi |u|^2 \dd x  \,=\, 
 \int_{\R^2} \phi u \cdot u_t\dd x \,=\, \int_{\R^2} \phi u \cdot 
 (\nu\Delta u - \nabla q)\dd x~.  
\end{equation}
Now, we have for $k = 1,2$\:
\[
  \int_{\R^2} |\nabla(\psi u_k)|^2 \,=\,  \int_{\R^2} |u_k\nabla\psi + 
  \psi\nabla u_k|^2 \,=\,  \int_{\R^2} (u_k^2|\nabla\psi|^2 + 
  \psi^2|\nabla u_k|^2 + 2\psi u_k\nabla\psi\cdot\nabla u_k)\dd x~, 
\]
hence 
\begin{align*}
  \int_{\R^2} \psi^2 u_k \Delta u_k \dd x \,&=\,
  -\int_{\R^2} \psi^2 |\nabla u_k|^2\dd x - 2\int_{\R^2} 
  \psi u_k \nabla\psi\cdot\nabla u_k\dd x \\
  \,&=\, -\int_{\R^2} |\nabla (\psi u_k)|^2\dd x + 
  \int_{\R^2} |\nabla\psi|^2 u_k^2\dd x~. 
\end{align*}
Thus summing over $k$ and multiplying by $\nu$, we obtain
\begin{equation}\label{locest2}
  \nu\int_{\R^2} \phi u\cdot (\Delta u) \dd x \,=\,
  -\nu\int_{\R^2} |\nabla (\psi u)|^2\dd x + \nu\int_{\R^2} 
  |\nabla\psi|^2 |u|^2\dd x~. 
\end{equation}
On the other hand, since $\div u = 0$, we easily find
\begin{equation}\label{locest3}
  \int_{\R^2} \phi u\cdot \nabla q \dd x \,=\, \int_{\R^2} \phi 
  \div(uq) \dd x \,=\, - \int_{\R^2} q(u\cdot\nabla\phi)\dd x~.
\end{equation}
Combining \eqref{locest1}--\eqref{locest3} we arrive at 
\eqref{locest}. 
\end{proof}

Our next task is to transform the identity \eqref{locest} into 
a differential inequality for the uniformly local energy norm
of the velocity field. The terms proportional to the viscosity 
parameter $\nu$ originate from the linear part of the equation 
and can be easily estimated, as in Section~\ref{ss3.2}. The 
difficulty is concentrated in the last term, which contains
the modified pressure $q = p + \half |u|^2$. That term is 
nonlocal in space and cubic in the velocity field $u$. Using
the results of Section~\ref{A2}, we obtain the following 
important estimate\:

\begin{lemma}\label{qloclem}
There exists a constant $C_4 > 0$ such that, if $x_0 \in \R^2$
and $0 < r \le R$, one has 
\begin{align}\nonumber
  \biggl|\int_{\R^2} q(u\cdot\nabla\phi_{R,x_0})\dd x\biggr| \,\le\, C_4
  \biggl(&\frac{r}{R}\|\omega\|_{L^\infty}\|u\|_{L^2(B_{x_0}^{3R})}^2 \\
  \label{qlocest} &+ \frac{1}{rR}\|u\|_{L^2(B_{x_0}^{3R})}^3 + \frac{1}{R^2}
  \sup_{z \in \R^2}\|u\|_{L^2(B_{z}^{2R})}^3\biggr)\,. 
\end{align}
\end{lemma}

\begin{proof}
For simplicity we assume that $x_0 = 0$, and we write $\phi_R$
instead of $\phi_{R,x_0}$. To use the results of Section~\ref{A2}
we choose a smooth function $\chi : \R^2 \to [0,1]$ satisfying
\[
  \chi(x) \,=\, \begin{cases} 1 & \hbox{if} \quad|x| \le 1/2\,, \\
  0 & \hbox{if} \quad|x| \ge 1\,, \end{cases}
\]
and we denote $\chi_r(x) = \chi(x/r)$, where $0 < r \le R$. To 
estimate the left-hand side of \eqref{qlocest}, it is sufficient
to control the modified pressure $q$ in the ball $B_0^{2R}$, 
because $\nabla \phi_R$ vanishes outside that ball. Also, as 
is clear from \eqref{locest3}, adding a constant to $q$ does 
not change the quantity we want to bound. Thus, using 
Lemma~\ref{qrepresentation}, we can assume that 
\begin{equation}\label{qdecomp}
  q(x) \,=\,  q_1(x) + q_2(x) + q_3(x) + q_4(x)~, 
  \qquad x \in B_0^{2R}~,
\end{equation}
where
\begin{align*}
  q_1(x) \,&=\, \frac{1}{2\pi}\int_{\R^2} \chi_r(x-y) \frac{(x-y)^\perp}{
  |x-y|^2}\cdot u(y)\omega(y)\dd y~, \\
  q_2(x) \,&=\, \frac{1}{4\pi}\sum_{k,\ell=1}^2\int_{\R^2} M_{k\ell}^{(r)}(x-y)
  u_k(y)u_\ell(y)\dd y~, \\ 
  q_3(x) \,&=\, \frac{1}{2\pi}\sum_{k,\ell=1}^2\int_{|y|\le 3R} 
  \chi_r^c(x-y)K_{k\ell}(x-y)u_k(y)u_\ell(y)\dd y~, \\ 
  q_4(x) \,&=\, \frac{1}{2\pi}\sum_{k,\ell=1}^2\int_{|y|\ge 3R} 
  \Bigl\{K_{k\ell}(x-y) - K_{k\ell}(x_0-y)\Bigr\}u_k(y)u_\ell(y)\dd y~. 
\end{align*}
The expressions above agree with \eqref{qidef} up to the following
inessential differences. First we use everywhere the rescaled cut-off
$\chi_r$ instead of $\chi$. In particular we denote by
$M_{k\ell}^{(r)}$ the functions defined by \eqref{MKdef} with $\chi$
replaced by $\chi_r$, and we write $\chi_r^c = 1 - \chi_r$.  Next, in
the definition \eqref{qidef} of $q_3(x,x_0)$, we take $x_0 = 0$ and we
decompose the domain of integration as $\R^2 = B_0^{3R} \cup
(B_0^{3R})^c$. The integral over $y \in B_0^{3R}$ coincides with the
function $q_3$ above, up to an irrelevant additive constant, and the
integral over $y \notin B_0^{3R}$ gives exactly $q_4$, because
$\chi_r^c(x-y) = \chi_r^c(-y) = 1$ when $|x| \le 2R$ and $|y| \ge 3R$.

We now estimate the various terms in \eqref{qdecomp}. As $\chi_r$ 
is supported in the ball $B_0^r \subset B_0^R$, we have
\[
  |q_1(x)| \,\le\,  \frac{1}{2\pi}\int_{|y|\le 3R} 
  \frac{\chi_r(x-y)}{|x-y|}\,|u(y)| |\omega(y)|\dd y~, 
  \qquad x \in B_0^{2R}~. 
\]
Since 
\[
  \int_{\R^2} \frac{\chi_r(z)}{|z|}\dd z \,=\,
  r\int_{\R^2} \frac{\chi(z)}{|z|}\dd z \,=\, Cr~, 
\] 
it follows from Young's inequality (see Section~\ref{A3}) that
\begin{equation}\label{q1est}
  \|q_1\|_{L^2(B_0^{2R})} \,\le\, Cr \|u\,\omega\|_{L^2(B_0^{3R})}
  \,\le\, C r \|\omega\|_{L^\infty} \|u\|_{L^2(B_0^{3R})}~.
\end{equation}
Similarly,
\[
  |q_2(x)| \,\le\,  \frac{1}{4\pi}\sum_{k,\ell=1}^2
  \int_{|y|\le 3R} |M_{k\ell}^{(r)}(x-y)| |u_k(y)| |u_\ell(y)| 
  \dd y~, \qquad x \in B_0^{2R}~. 
\]
As $|M_{k\ell}^{(r)}(z)| \le C |z|^{-1} |\nabla \chi_r(z)| = 
C r^{-1} |z|^{-1} |\nabla \chi(z/r)|$, we have
\[
  \sum_{k,\ell=1}^2\int_{\R^2}|M_{k\ell}^{(r)}(z)|^2\dd z \,\le\, 
  \frac{C}{r^2}\int_{r/2 \le |z|\le r}\frac{1}{|z|^2}\dd z 
  \,=\, \frac{C}{r^2}~,
\]
and using Young's inequality again we find
\begin{equation}\label{q2est}
  \|q_2\|_{L^2(B_0^{2R})} \,\le\, \frac{C}{r}\,\||u|^2\|_{L^1(B_0^{3R})}
  \,=\, \frac{C}{r}\,\|u\|_{L^2(B_0^{3R})}^2~.
\end{equation}
Exactly the same estimate holds for $q_3$ too, because in view
of \eqref{MKdef}
\[
  \sum_{k,\ell=1}^2\int_{\R^2}\chi_r^c(z)^2 K_{k\ell}(z)^2\dd z 
  \,\le\, C \int_{|z|\ge r/2}\frac{1}{|z|^4}\dd z 
  \,=\, \frac{C}{r^2}~.
\]
Finally, to bound $q_4$, we use estimate \eqref{Kest} which 
shows that, for any $x \in B_0^{2R}$ and $y \notin B_0^{3R}$,
\[
  \sum_{k,\ell=1}^2 |K_{k\ell}(x-y) - K_{k\ell}(-y)| \,\le\, 
  \frac{CR}{2|y|^3} \,\le\, \frac{CR}{R^3 + |y|^3}~.
\]
Thus
\[
  |q_4(x)| \,\le\, C\int_{\R^2} \frac{R}{R^3 + |y|^3}\,
  |u(y)|^2 \dd y~, \qquad x \in B_0^{2R}~.
\]
To evaluate the integral we decompose $\R^2 = \cup_{k \in \Z^2} Q_k^R$, 
where $Q_k^R \subset B_{kR}^R$ denotes the square of measure $R^2$ 
centered at $kR \in \R^2$. For $k \in \Z^2$ we also define 
\[
  S_k \,=\, \sup_{y \in Q_k^R}\,\frac{R}{R^3 + |y|^3}
  \,\le\,  \frac{C}{R^2}\,\frac{1}{1+|k|^3}~, \qquad 
  \hbox{so that}\quad \sum_{k\in\Z^2} S_k \,\le\, \frac{C}{R^2}~.
\]
With these notations we find for $x \in B_0^{2R}$\: 
\begin{align*}
  |q_4(x)| \,\le\, C\sum_{k \in \Z^2}\int_{Q_k^R} \frac{R}{R^3 + |y|^3}\,
  |u(y)|^2 \dd y \,\le\, C\biggl(\sum_{k \in \Z^2} S_k\biggr) 
  \sup_{k \in \Z^2}\|u\|_{L^2(Q_k^R)}^2~, 
\end{align*}
hence
\begin{equation}\label{q4est}
  \|q_4\|_{L^2(B_0^{2R})} \,\le\, 2\pi^{1/2}R \|q_4\|_{L^\infty(B_0^{2R})} 
  \,\le\, \frac{C}{R}\,\sup_{z \in \R^2}\|u\|_{L^2(B_z^R)}^2~. 
\end{equation}

Summarizing, estimates \eqref{q1est}--\eqref{q4est} show that 
\begin{equation}\label{q5est}
  \|q\|_{L^2(B_0^{2R})} \,\le\, C r \|\omega\|_{L^\infty} \|u\|_{L^2(B_0^{3R})}
  + \frac{C}{r}\,\|u\|_{L^2(B_0^{3R})}^2 + \frac{C}{R}\,\sup_{z \in \R^2}
  \|u\|_{L^2(B_z^R)}^2~. 
\end{equation}
On the other hand, as $\nabla \phi_R(x) = R^{-1}\nabla\phi(x/R)$, 
we have by H\"older's inequality
\[
  \biggl|\int_{\R^2} q(u\cdot\nabla\phi_R)\dd x\biggr| \,\le\,
  \frac{C}{R}\,\|u\|_{L^2(B_0^{2R})}\|q\|_{L^2(B_0^{2R})}~,
\]
thus using \eqref{q5est} we easily obtain estimate \eqref{qlocest}.
\end{proof}

Combining Lemmas~\ref{loclem} and \ref{qloclem}, we now derive 
an integral inequality for the quantity
\begin{equation}\label{Zdef}
  Z_R(t) \,=\, \sup_{x \in \R^2} \|u(t)\|_{L^2(B_x^R)}~,
\end{equation}
which is equivalent (up to $R$-dependent constants) to 
the norm $\|u(t)\|_{L^2_\ul}$. 

\begin{lemma}\label{Zint}
There exists a constant $C_5 \ge 1$ such that, for any $t > 0$ 
and any $R \ge r > 0$, 
\begin{align}\nonumber
  Z_R(t)^2 &+ 2\nu \sup_{x\in\R^2}\int_0^t \|\nabla u(s)\|_{L^2(B_x^R)}^2\dd s
  \,\le\, 7 Z_R(0)^2 \\ \label{Zest} 
  &+ C_5 \int_0^t \Bigl\{\frac{\nu}{R^2}\,Z_R(s)^2 + \frac{r}{R} 
  \|\omega(s)\|_{L^\infty} Z_R(s)^2 + \frac{1}{rR} Z_R(s)^3\Bigr\}\dd s~.
\end{align}
\end{lemma}

\begin{proof}
Fix $R > 0$. Integrating \eqref{locest} with respect to time, we obtain 
for any $t > 0$\:
\begin{align}\nonumber
  \int_{\R^2} \phi_{R,x_0} &|u(t)|^2 \dd x +  
  2\nu\int_0^t\!\int_{\R^2} |\nabla(\phi_{R,x_0}^{1/2}u(s))|^2 \dd x 
  \dd s \,=\, \int_{\R^2} \phi_{R,x_0} |u_0|^2 \dd x \\ \label{intphi}
  &+ 2\nu\int_0^t\!\int_{\R^2} |\nabla \phi_{R,x_0}^{1/2}|^2 |u(s)|^2 \dd x 
  \dd s + 2\int_0^t\!\int_{\R^2} q(s) (u(s)\cdot \nabla\phi_{R,x_0})\dd x
  \dd s~,   
\end{align}
where $x_0 \in \R^2$ is arbitrary. We now take the supremum over
$x_0 \in \R^2$ in both sides. As $\phi_{R,x_0} = 1$ on $B_{x_0}^R$, 
the supremum over $x_0\in \R^2$ of the left-hand side of 
\eqref{intphi} is bounded from below by the left-hand side of 
\eqref{Zest}. To bound the right-hand side of \eqref{intphi}, 
we observe that
\[
  \sup_{x_0\in\R^2} \int_{\R^2} \phi_{R,x_0} |u_0|^2 \dd x \,\le\, 
  \sup_{x_0\in\R^2} \|u_0\|_{L^2(B_{x_0}^{2R})}^2 \,\le\, 7 Z_R(0)^2~,
\]
because each ball $B_{x_0}^{2R}$ can be covered by $7$ balls of radius 
$R$, centered at appropriate points. Similarly, as $|\nabla 
\phi_{R,x_0}^{1/2}|$ is bounded by $C/R$ and vanishes outside 
$B_{x_0}^{2R}$, we have  
\[
  \sup_{x_0\in\R^2}\int_0^t\!\int_{\R^2} |\nabla \phi_{R,x_0}^{1/2}|^2 
  |u(s)|^2 \dd s \,\le\, \frac{C}{R^2} \sup_{x_0\in\R^2}\int_0^t 
  \|u(s)\|_{L^2(B_{x_0}^{2R})}^2 \dd s \,\le\, \frac{7C}{R^2}
  \int_0^t Z_R(s)^2\dd s~.
\]
Finally, for the last term in \eqref{intphi}, we choose $r \in (0,R]$
and we use inequality \eqref{qlocest}. Collecting all estimates, 
we see that the supremum over $x_0 \in \R^2$ of the right-hand side of 
\eqref{intphi} is bounded by the right-hand side of \eqref{Zest}, 
provided $C_5 > 0$ is large enough.
\end{proof}

Using Lemma~\ref{Zint} and Gronwall's lemma, we arrive at 
the main result of this section. 

\begin{proposition}\label{Zprop}
There exist positive constants $C_6$ and $C_7$ such that the following holds 
for any $\nu > 0$. Let $u \in C^0([0,+\infty),X)$ be the mild solution 
of the Navier-Stokes equations \eqref{2DNS} with initial data $u_0 \in X$
satisfying $\div u_0 = 0$, $\omega_0 \in L^\infty(\R^2)$, and 
$\omega_0 \not\equiv 0$. For any $t > 0$, if $R > 0$ is 
large enough so that
\begin{equation}\label{Rdef}
  R \,\ge\, \max\Bigl\{R_0\,,\,C_7 \sqrt{\nu t}\,,\,C_7 \|u_0\|_{L^\infty}
 \|\omega_0\|_{L^\infty}t^2\Bigr\}~, \qquad \hbox{where}\quad R_0 \,=\, 
 \frac{\|u_0\|_{L^\infty}}{\|\omega_0\|_{L^\infty}}~,
\end{equation}
we have 
\begin{equation}\label{Zpropest}
  Z_R(t) \,\equiv\, \sup_{x \in \R^2} \|u(t)\|_{L^2(B_x^R)} 
  \,\le\, C_6 R \|u_0\|_{L^\infty}~.
\end{equation}
\end{proposition}

\begin{proof} 
We observe that $u_0 \not\equiv 0$, since by assumption
$\omega_0 \not\equiv 0$. Take $C_6 > 0$ and $C_7 > 0$ such that
\begin{equation}\label{Zprop1}
  C_6^2 \,>\,  7\e\pi~, \qquad C_7^2 \,\ge\, 2C_5~, \qquad
  C_7^{1/2} \,\ge\, 2C_5(1+C_6)~,
\end{equation}
where $C_5$ is as in Lemma~\ref{Zint}. Given any $t > 0$, choose $R$
as in \eqref{Rdef}. From the definition \eqref{Zdef}, we see that
$Z_R(0) \le \pi^{1/2} R \|u_0\|_{L^\infty}$, hence by continuity we
necessarily have $Z_R(s) \le C_6R\|u_0\|_{L^\infty}$ for sufficiently
small $s > 0$. Define
\begin{equation}\label{Zprop2}
  t_* \,=\, \sup\Bigl\{\tau \in [0,t]\,\Big|\, Z_R(s) \le C_6R\|u_0\|_{L^\infty}
  \hbox{ for all } s \in [0,\tau]\Bigr\} \,\in\, (0,t]~.
\end{equation}
We shall prove that $t_* = t$, hence $Z_R(t) \le C_6R\|u_0\|_{L^\infty}$, 
which is \eqref{Zpropest}. 

According to Lemma~\ref{Zint}, we have for all $\tau \in [0,t_*]$\:
\begin{equation}\label{Zprop3}
  Z_R(\tau)^2 \,\le\, 7 Z_R(0)^2 +  C_5 \int_0^\tau A(s)Z_R(s)^2\dd s~,
\end{equation}
where
\[ 
  A(s) \,=\, \frac{\nu}{R^2} + \frac{r}{R}\|\omega(s)\|_{L^\infty} + 
  \frac{1}{rR}\,Z_R(s) \,\le\, \frac{\nu}{R^2} + \frac{r}{R}
  \|\omega_0\|_{L^\infty} + \frac{C_6}{r}\|u_0\|_{L^\infty}~.
\]
Here $r \le R$ is arbitrary, but we can optimize the right-hand 
side by choosing $r = (RR_0)^{1/2}$, which gives
\[
  A(s) \,\le\,  \frac{\nu}{R^2} + (1+C_6)\frac{\|u_0\|_{L^\infty}^{1/2}
  \|\omega_0\|_{L^\infty}^{1/2}}{R^{1/2}}~, \qquad s \in [0,t_*]~. 
\]
In particular, using \eqref{Rdef}, \eqref{Zprop1} and the fact that 
$t_* \le t$, we find
\[
  C_5\int_0^{t_*} A(s)\dd s \,\le\, C_5 \frac{\nu t}{R^2} + 
  C_5(1+C_6)\frac{\|u_0\|_{L^\infty}^{1/2}\|\omega_0\|_{L^\infty}^{1/2} 
  t}{R^{1/2}} \,\le\, \frac12 + \frac12 \,=\, 1~.
\]
If we now apply Gronwall's lemma (see Section~\ref{A3}) 
to \eqref{Zprop3}, we obtain
\begin{equation}\label{Zfinal}
  Z_R(t_*)^2 \,\le\, 7 Z_R(0)^2 \exp\biggl(C_5\int_0^{t_*}
  A(s)\dd s\biggr) \,\le\, 7 \e Z_R(0)^2 \,\le\, 7 \e \pi 
  R^2 \|u_0\|_{L^\infty}^2~.
\end{equation}
By \eqref{Zprop1} we thus have $Z_R(t_*) < C_6  R \|u_0\|_{L^\infty}$, 
which contradicts \eqref{Zprop2} if $t_* < t$. Thus we must have
$t_* = t$, which proves \eqref{Zpropest}. 
\end{proof}

\begin{remark}
Estimate \eqref{ulER} follows immediately from \eqref{Zfinal} with
$t_* = t$. 
\end{remark}

\subsection{Velocity bounds and uniformly local enstrophy 
estimates}\label{ss3.4}

In this final section, we derive a few important consequences of 
the previous results. First, combining Proposition~\ref{Zprop}
with Corollary~\ref{BScor}, we derive an upper bound on the $L^\infty$ 
norm of the velocity field which greatly improves \eqref{globexbd}. 

\begin{proposition}\label{bestglob}
There exist a positive constant $C_8$ such that, for any $u_0 \in X$ 
satisfying $\div u_0 = 0$ and $\omega_0 = \curl u_0\in L^\infty(\R^2)$, 
the solution of the Navier-Stokes equations \eqref{2DNS} with initial 
data $u_0$ satisfies 
\begin{equation}\label{bestglobex}
  \|u(t)\|_{L^\infty} \,\le\, C_8\|u_0\|_{L^\infty}\biggl\{1 + 
  \|\omega_0\|_{L^\infty}t + \Bigl(\frac{\sqrt{\nu t}}{R_0}\Bigr)^{1/2}
  \biggr\}~, \qquad t \ge 0~,
\end{equation}
where $R_0 = \|u_0\|_{L^\infty}/\|\omega_0\|_{L^\infty}$. 
\end{proposition}

\begin{proof}
If $\omega_0 \equiv 0$, then $u_0$ is a constant and $u(t) = u_0$ for
all $t \ge 0$, hence we can assume without loss of generality that
$\omega_0 \not\equiv 0$. Fix $t > 0$, and let $R = \max\{R_0\,,\,C_7
\sqrt{\nu t}\,,\,C_7 \|u_0\|_{L^\infty} \|\omega_0\|_{L^\infty}t^2\}$,
as in \eqref{Rdef}. If $M = \|u(t)\|_{L^\infty} > 0$, there exists
$\bar x \in \R^2$ such that $|u(\bar x,t)| \ge M/2$. For simplicity,
we assume without loss of generality that $\bar x = 0$.  We know from
Proposition~\ref{Zprop} that
\begin{equation}\label{add1}
  I := \int_{|x| \le R} |u(x,t)|^2 \dd x \,\le\, C_6^2 R^2 
  \|u_0\|_{L^\infty}^2~.
\end{equation}
On the other hand, applying Corollary~\ref{BScor} to $u(x,t)$, we 
deduce from \eqref{BS5} that
\begin{equation}\label{add2}
  |u(x,t) + u(-x,t)| \,\ge\, 2|u(0,t)| - C_*|x| 
  \|\omega(t)\|_{L^\infty} \,\ge\, M - C_*|x| \|\omega_0\|_{L^\infty}~, 
\end{equation}
for all $x \in \R^2$, where $C_* > 0$ is a universal constant. 
The idea is to use estimate \eqref{add2} to obtain a lower bound
on the quantity $I$ defined in \eqref{add1}, in terms of $M$. 
Let 
\[
  R_* \,=\, \frac{M}{C_*\|\omega_0\|_{L^\infty}}~.
\]
We consider separately the following two cases\:

\smallskip\noindent{\bf Case 1\:} $R_* \le R$. In that case, 
denoting $D_+ = \{x \in \R^2\,|\, |x| \le R_*\,,~x_1 \ge 0\}$, we
compute
\begin{align*}
  I \,&\ge\, \int_{|x|\le R_*} |u(x,t)|^2 \dd x \,=\, 
  \int_{D_+} \Bigl(|u(x,t)|^2 + |u(-x,t)|^2\Bigr)\dd x \\ 
  \,&\ge\, \frac12 \int_{D_+} |u(x,t) + u(-x,t)|^2\dd x 
  \,\ge\, \frac14 \int_{|x|\le R_*} \Bigl(M - C_*|x| 
  \|\omega_0\|_{L^\infty}\Bigr)^2\dd x \,=\, 
  \frac{\pi}{24}\,M^2 R_*^2~,
\end{align*}
where in the last inequality we used \eqref{add2}. Comparing
with \eqref{add1}, we deduce that
\[
  M^2 R_*^2 \,\le\, C  R^2 \|u_0\|_{L^\infty}^2~, \qquad 
  \hbox{or}\quad M^4 \,\le\, C  R^2 \|u_0\|_{L^\infty}^2
  \|\omega_0\|_{L^\infty}^2~.
\]
Since $M = \|u(t)\|_{L^\infty}$ and $R = \max\{R_0\,,\,C_7 \sqrt{\nu t}\,,
\,C_7 \|u_0\|_{L^\infty}\|\omega_0\|_{L^\infty}t^2\}$, this gives 
\eqref{bestglobex}. 

\smallskip\noindent{\bf Case 2\:} $R_* \ge R$. A similar 
calculation gives
\[
  I \,\ge\, \frac14 \int_{|x|\le R} \Bigl(M - C_*|x| 
  \|\omega_0\|_{L^\infty}\Bigr)^2\dd x  \,\ge\, \frac{\pi}{24}\,M^2 R^2~,
\]
hence $M^2 R^2 \le C R^2 \|u_0\|_{L^\infty}^2$, namely $M^2 \le C 
\|u_0\|_{L^\infty}^2$. Thus we obtain \eqref{bestglobex} in both cases. 
\end{proof}

\medskip\noindent{\em Proof of Theorem~\ref{thm2}.} We can assume
without loss of generality that $u_0 \not\equiv 0$. Since global
existence of solutions is already asserted by Theorem~\ref{thm1}, we
only need to prove estimate \eqref{thm2est}. In view of the local
existence theory, it is sufficient to prove \eqref{thm2est} for $t \ge
T$, where $T > 0$ is as in \eqref{smoothT}. In that case we have
$\sqrt{\nu t} \le K_1\|u_0\|_{L^\infty}t$, or equivalently
\[
  \frac{\sqrt{\nu t}}{R_0} \,=\, \frac{\|\omega_0\|_{L^\infty}}{
  \|u_0\|_{L^\infty}}\,\sqrt{\nu t} \,\le\, K_1\|\omega_0\|_{L^\infty}t~,
\]
hence \eqref{thm2est} follows directly from \eqref{bestglobex}. \QED

\medskip
The results obtained so far do not rely on the viscous dissipation 
term in \eqref{2DNS}, and remain therefore valid in the vanishing 
viscosity limit. Now, assuming that $\nu > 0$, we can also 
derive uniformly local enstrophy estimates.

\begin{proposition}\label{enstrophy}
Under the assumptions of Proposition~\ref{Zprop}, there exists a 
constant $C_9 > 0$ such that, for all $t > 0$,  
\begin{equation}\label{enstrophyest}
  \sup_{x \in \R^2} \|\omega(t)\|_{L^2(B_x^R)}^2\dd y \,\le\,   
  C_9 \|u_0\|_{L^2}^2 \Bigl(1 + \frac{R^2}{\nu t} + \frac{Rt}{\sqrt{\nu t}}
  \,\|\omega_0\|_{L^\infty}\Bigr)~,
\end{equation}
where $R = R(t)$ is as in \eqref{Rdef}.
\end{proposition}

\begin{proof}
Fix $t > 0$ and let $R$ be as in \eqref{Rdef}. If one does not 
neglect the second term in the right-hand side of \eqref{Zest}, 
the proof of Proposition~\ref{Zprop} shows that
\begin{equation}\label{enst1}
  Z_R(t)^2 + 2\nu \sup_{x \in \R^2} \int_0^t \|\nabla u(s)\|_{L^2(B_x^R)}^2
  \dd s \,\le\, C_6^2 R^2 \|u_0\|_{L^\infty}^2~. 
\end{equation}
Unfortunately, we cannot extract from \eqref{enst1} a pointwise
estimate in time on the uniformly local $L^2$ norm of $\nabla u$, 
because we cannot exchange the supremum and the integral in the
left-hand side of \eqref{enst1}. 

To avoid that difficulty, we use localized energy estimates for the
vorticity equation \eqref{omeq}. Let $x_0 \in \R^2$. As in 
Lemma~\ref{loclem} we have 
\begin{align}\nonumber
  \frac12 \frac{\D}{\D t}\int_{\R^2} \phi_{R,x_0} |\omega|^2 \dd x &+ 
  \nu\int_{\R^2} |\nabla(\phi_{R,x_0}^{1/2}\omega)|^2 \dd x \\ \label{locom}
  \,&=\, \nu\int_{\R^2} |\nabla \phi_{R,x_0}^{1/2}|^2 |\omega|^2 \dd x +
  \frac12 \int_{\R^2} \omega^2 (u\cdot \nabla\phi_{R,x_0})\dd x~,   
\end{align}
where $\phi_{R,x_0}$ is the localization function defined in 
\eqref{phiRdef}. In view of \eqref{enst1}, there exists a time 
$t_0 \in [0,t/2]$ (depending on $x_0$) such that
\begin{equation}\label{enst2}
  \nu\|\omega(t_0)\|_{L^2(B_{x_0}^{2R})}^2 \,\le 2\nu\|\nabla u(t_0)
  \|_{L^2(B_{x_0}^{2R})}^2 \,\le\, \frac{C R^2}{t}\,\|u_0\|_{L^\infty}^2~.
\end{equation}
Integrating \eqref{locom} over the time interval $[t_0,t]$, we
find
\begin{align}\nonumber
  \int_{\R^2} \phi_{R,x_0} |\omega(t)|^2 \dd x \,\le\,  
  &\int_{\R^2} \phi_{R,x_0} |\omega(t_0)|^2 \dd x 
  + 2\nu\int_{t_0}^t\!\int_{\R^2} |\nabla \phi_{R,x_0}^{1/2}|^2 
  |\omega(s)|^2 \dd x \dd s \\ \label{intom}
  &+ \int_{t_0}^t\!\int_{\R^2} \omega(s)^2 
  (u(s)\cdot \nabla\phi_{R,x_0})\dd x \dd s~.   
\end{align}
By \eqref{enst2} we have
\begin{equation}\label{enst3}
  \int_{\R^2} \phi_{R,x_0} |\omega(t_0)|^2 \dd x \,\le\, 
  \|\omega(t_0)\|_{L^2(B_{x_0}^{2R})}^2 \,\le\, \frac{C R^2}{\nu t}\, 
  \|u_0\|_{L^\infty}^2~.
\end{equation}
It follows also from \eqref{enst1} that
\begin{equation}\label{enst4}
  2\nu\int_{t_0}^t\!\int_{\R^2} |\nabla \phi_{R,x_0}^{1/2}|^2 
  |\omega(s)|^2 \dd x \dd s \,\le\, \frac{C\nu}{R^2}
  \int_0^t \|\omega(s)\|_{L^2(B_{x_0}^{2R})}^2 \dd s 
  \,\le\, C \|u_0\|_{L^\infty}^2~.
\end{equation}
Finally, the cubic term in \eqref{intom} can be estimated 
as follows\:
\begin{align}\nonumber
  \Bigl|\int_{t_0}^t\!\int_{\R^2} \omega(s)^2 &(u(s)\cdot \nabla\phi_{R,x_0})
  \dd x \dd s\Bigr| \\ \nonumber 
  \,&\le\, \frac{C}{R} \int_0^t \|\omega(s)\|_{L^\infty}
  \|\omega(s)\|_{L^2(B_{x_0}^{2R})} \|u(s)\|_{L^2(B_{x_0}^{2R})}\dd s \\ \label{enst5}
  \,&\le\, \frac{C}{R}\,\|\omega_0\|_{L^\infty} \Bigl(\int_0^t 
  \|\omega(s)\|_{L^2(B_{x_0}^{2R})}^2\dd s\Bigr)^{1/2} t^{1/2} 
  \sup_{0 \le s \le t} \|u(s)\|_{L^2(B_{x_0}^{2R})} \\ \nonumber
  \,&\le\, \frac{C}{R}\,\|\omega_0\|_{L^\infty} \cdot \frac{CR
  \|u_0\|_{L^\infty}}{\nu^{1/2}}\cdot t^{1/2}\cdot CR\|u_0\|_{L^\infty} 
  \,\le\, \frac{CRt}{\sqrt{\nu t}}\,\|\omega_0\|_{L^\infty} 
  \|u_0\|_{L^\infty}^2~.
\end{align}
If we now insert \eqref{enst3}--\eqref{enst5} into \eqref{intom} 
we obtain for all $x_0 \in \R^2$\:
\[
  \int_{\R^2} \phi_{R,x_0} |\omega(t)|^2 \dd x \,\le\, C 
  \|u_0\|_{L^2}^2 \Bigl(1 + \frac{R^2}{\nu t} + \frac{Rt}{\sqrt{\nu t}}
  \,\|\omega_0\|_{L^\infty}\Bigr)~.   
\]
Taking the supremum over $x_0\in \R^2$, we arrive at 
\eqref{enstrophyest}. 
\end{proof}

\begin{remark}\label{remfin}
Estimate \eqref{ulom} easily follows from \eqref{enstrophyest}. 
Indeed, in view of \eqref{smoothT}, it is clearly sufficient to 
prove \eqref{ulom} for $t \ge T$. In that case we have 
$\sqrt{\nu t} \le K_1\|u_0\|_{L^\infty}t$, hence 
\[
  \frac{Rt}{\sqrt{\nu t}}\,\|\omega_0\|_{L^\infty} \,=\, 
  \frac{R}{\nu}\,\|\omega_0\|_{L^\infty}\sqrt{\nu t} \,\le\, 
  \frac{K_1R}{\nu t}\,\|u_0\|_{L^\infty}\|\omega_0\|_{L^\infty}t^2 
  \,\le\, \frac{K_1}{C_7}\frac{R^2}{\nu t}~,
\]
where in the last inequality we used definition \eqref{Rdef}, 
which also implies that $R^2 \ge C_7^2 \nu t$. Thus 
\eqref{ulom} follows from \eqref{enstrophyest} when $t \ge T$. 
\end{remark}

\section{Appendix}
\label{s4}

\subsection{The Biot-Savart law for bounded velocities and 
vorticities}
\label{A1}

Assume that $u \in X$ satisfies $\div u = 0$, and let $\omega = 
\curl u = \partial_1 u_2 - \partial_2 u_1$.  If $\omega$ is strongly
localized, for instance if $\omega \in L^p(\R^2)$ for some $p \in
(1,2)$, then $u$ can be reconstructed from $\omega$, up to an additive
constant $u_\infty$, by the classical Biot-Savart formula
(see Section~\ref{A3})\:
\begin{equation}\label{BS1}
  u(x) \,=\, u_\infty + \frac{1}{2\pi}\int_{\R^2} 
  \frac{(x-y)^\perp}{|x-y|^2}\,\omega(y)\dd y~, \qquad 
  x \in \R^2~.
\end{equation}
Moreover, the Hardy-Littlewood-Sobolev inequality \cite{LL} implies
that $u - u_\infty \in L^q(\R^2)^2$ for $q = 2p/(2-p)$, hence $u(x)$
converges in some sense to $u_\infty$ as $|x| \to \infty$.

If the vorticity distribution $\omega$ is only weakly localized, 
the integral in \eqref{BS1} does not converge any more, and 
the Biot-Savart formula has to be modified. Here is a 
reasonable possibility\:

\begin{proposition}\label{BSweak}
If $\omega \in L^p(\R^2)$ for some $p \in (2,\infty)$, the 
velocity field $u \in X$ satisfies 
\begin{equation}\label{BS2}
  u(x) \,=\, u(0) + \frac{1}{2\pi}\int_{\R^2} 
  \biggl\{\frac{(x-y)^\perp}{|x-y|^2} + \frac{y^\perp}{|y|^2}
  \biggr\}\,\omega(y)\dd y~, \qquad x \in \R^2~.
\end{equation}
\end{proposition}

\begin{proof}
Let $q = p/(p-1)$, so that $q \in (1,2)$ and $\frac1p + \frac1q = 1$. 
For all $x,y \in \R^2$ with $x \neq y$ and $y \neq 0$, we denote
\begin{equation}\label{Fxydef}
  F(x,y) \,=\, \frac{(x-y)^\perp}{|x-y|^2} + \frac{y^\perp}{|y|^2} \,=\, 
  \frac{x^\perp(y\cdot(y{-}x)) + y^\perp(x\cdot(x{-}y)) + (x^\perp{-}y^\perp)
  (x\cdot y)}{|x-y|^2 |y|^2}~.
\end{equation}
We claim that, for any $x \in \R^2$, the map $y \mapsto F(x,y)$ belongs 
to $L^q(\R^2)$ and
\begin{equation}\label{Fqbound}
  \biggl(\int_{\R^2} |F(x,y)|^q \dd y\biggr)^{1/q} \,\le\, C 
  |x|^{\frac2q - 1}~,  
\end{equation}
where $C > 0$ is a universal constant. Indeed, as $|F(x,y)| \le 
|x-y|^{-1} + |y|^{-1}$, we obtain using Minkowski's inequality
\[
  \biggl(\int_{|y|\le 2|x|} |F(x,y)|^q \dd y\biggl)^{1/q} \,\le\, 
  2 \biggl(\int_{|y|\le 3|x|} \frac{1}{|y|^q}\dd y\biggl)^{1/q} 
  \,\le\, C |x|^{\frac2q - 1}~.  
\]
On the other hand, since $|F(x,y)| \le 3|x| |x-y|^{-1} |y|^{-1}$
in view of the last expression in \eqref{Fxydef}, we have 
$|F(x,y)| \le 6|x||y|^{-2}$ when $|y| \ge 2|x|$, hence
\[
  \biggl(\int_{|y| \ge 2|x|} |F(x,y)|^q \dd y\biggl)^{1/q} \,\le\, 
  6 |x|\biggl(\int_{|y|\ge 2|x|} \frac{1}{|y|^{2q}}\dd y\biggl)^{1/q} 
  \,\le\, C |x|^{\frac2q - 1}~.  
\]
This proves \eqref{Fqbound}. Now, let
\[
  v(x) \,=\, \frac{1}{2\pi}\int_{\R^2} F(x,y)\omega(y)\dd y \,=\, 
  \frac{1}{2\pi}\int_{\R^2} \biggl\{\frac{(x-y)^\perp}{|x-y|^2} + 
  \frac{y^\perp}{|y|^2}\biggr\}\,\omega(y)\dd y~, \qquad x \in \R^2~.
\]
Since $\frac2q - 1 = 1 - \frac2p$, we deduce from \eqref{Fqbound} that
$|v(x)| \le C|x|^{1-\frac2p}\|\omega\|_{L^p}$ for all $x \in \R^2$.
Moreover, a standard calculation in distribution theory shows that
$\div v = 0$ and $\curl v = \omega$. If $w = u - v$, we thus have
$\div w = 0$ and $\curl w = 0$, so that $w$ is a harmonic vector field
on $\R^2$. As $w(x)$ has a sublinear growth as $|x| \to \infty$, we
conclude that $w$ is identically constant, and since $v(0) = 0$ by
definition we must have $w = u(0)$, which proves \eqref{BS2}.
\end{proof}

In these notes, we mainly deal with the situation where the 
vorticity $\omega$ is not localized at all, namely $\omega \in
L^\infty(\R^2)$. In that case, the integral in \eqref{BS2} is
logarithmically divergent, and has to be interpreted in an 
appropriate way. The main result of this section is\:

\begin{proposition}\label{BSultim}
Assume that $u \in X$, $\div u = 0$, and $\omega = \curl u \in 
L^\infty(\R^2)$. Then
\begin{equation}\label{BS3}
  u(x) \,=\, u(0) + \lim_{R \to \infty}\frac{1}{2\pi}\int_{|y|\le R} 
  \biggl\{\frac{(x-y)^\perp}{|x-y|^2} + \frac{y^\perp}{|y|^2}
  \biggr\}\,\omega(y)\dd y~, \qquad x \in \R^2~.
\end{equation}
\end{proposition}

\begin{proof}
For any $R > 0$ we denote
\[
  u_R(x) \,=\, \frac{1}{2\pi}\int_{|y|\le R} 
  \biggl\{\frac{(x-y)^\perp}{|x-y|^2} + \frac{y^\perp}{|y|^2}
  \biggr\}\,\omega(y)\dd y~, \qquad x \in \R^2~.
\]
Clearly $\div u_R = 0$ and $\curl u_R = \omega\,\1_{\{|x|\le R\}}$ in
the sense of distributions on $\R^2$. Moreover $u_R$ is H\"older 
continuous and satisfies $u_R(0) = 0$. For $R > 2|x| + 3$, we 
decompose $u_R(x) = v(x) + w_R(x)$, where
\begin{equation}\label{vwdef}
  v(x) \,=\, \frac{1}{2\pi}\int_{D_x} F(x,y)\omega(y)\dd y~, 
  \qquad w_R(x) \,=\, \frac{1}{2\pi}\int_{B_0^R \setminus D_x} F(x,y)
  \omega(y)\dd y~. 
\end{equation}
Here the following notations have been used. As in Section~\ref{s3}, 
we denote by $B_x^r$ the closed ball of radius $r \ge 0$ centered at 
$x \in \R^2$. For any $x \in \R^2$, we define
\[
  D_x \,=\, \begin{cases} B_0^3 & \hbox{if} \quad|x| \le 2\,, \\
  B_0^1 \cup B_x^1 & \hbox{if} \quad |x| > 2\,, \end{cases}
\] 
so that $D_x \subset B_0^R$ and $D_x$ has smooth boundary. Finally 
$F(x,y)$ is as in \eqref{Fxydef}. 

We now estimate both terms in \eqref{vwdef}. As $|F(x,y)| \le 
|x-y|^{-1} + |y|^{-1}$, we have
\[
  |v(x)| \,\le\, \frac{1}{\pi}\int_{|y|\le 5} \frac{1}{|y|}|\omega(y)|
  \dd y \,\le\, 10\|\omega\|_{L^\infty}~, \qquad \hbox{if } |x| \le 2~,
\]
and
\[
  |v(x)| \,\le\, \frac{2}{\pi}\int_{|y|\le 1} \frac{1}{|y|}|\omega(y)|
  \dd y \,\le\, 4\|\omega\|_{L^\infty}~, \qquad \hbox{if } |x| > 2~.
\]
Moreover $v$ is continuous and $v(0) = 0$. To bound $w_R$, we use the
fact that $\omega = \curl u = -\div u^\perp$, and we integrate by
parts using Green's formula. We obtain
\begin{align*}
  w_R(x) \,&=\, \frac{1}{2\pi}\int_{\partial D_x}\,F(x,y)u^\perp(y)
  \cdot\nu(y) \dd\ell_y -\frac{1}{2\pi}\int_{\partial B_0^R}\,F(x,y)u^\perp(y)
  \cdot\nu(y) \dd\ell_y \\ &\quad+\, 
  \frac{1}{2\pi}\int_{B_0^R \setminus D_x}u^\perp(y)\cdot\nabla_y
  F(x,y)\dd y \,=\, w^{(1)}(x) -  w_R^{(2)}(x) +  w_R^{(3)}(x)~,
\end{align*}
where on the circles $\partial D_x$ or $\partial B_0^R$ we denote
by $\nu$ the exterior unit normal and $\D\ell$ the elementary 
arc length. Proceeding as in the proof of Proposition~\ref{BSweak}, 
it is straightforward to verify that $| w^{(1)}(x)| \le 4\|u\|_{L^\infty}$ 
and $|w_R^{(2)}(x)| \le 6|x|\|u\|_{L^\infty}/R$. In particular $w_R^{(2)}(x)$ 
converges to zero as $R \to \infty$, for any $x \in \R^2$. Finally, 
using the estimate
\begin{equation}\label{nablaFest}
  |\nabla_y F(x,y)| \,\le\, C\biggl(\frac{|x|}{|x-y|^2 |y|} 
  + \frac{|x|}{|x-y|\,|y|^2}\biggl)~, \qquad x \neq y~, \quad 
  y \neq 0~,
\end{equation}
which can be obtained by a direct calculation, it is not difficult 
to show that
\begin{equation}\label{nablaFint}
  \frac{1}{2\pi}\int_{\R^2 \setminus D_x} |u(y)| |\nabla_y F(x,y)|\dd y
  \,\le\, C \|u\|_{L^\infty}\log(1+|x|)~,
\end{equation}
for some universal constant $C > 0$. When evaluating that integral for
large $|x|$, it is convenient to consider separately the regions where
$1 \le |y| \le |x|/2$, where $1 \le |y-x| \le |x|/2$, where $|y| \ge
2|x|$, and the region where $|x|/2 \le |y| \le 2|x|$ with 
$|y-x| \ge |x|/2$. Estimate \eqref{nablaFint} implies in particular 
that $w_R^{(3)}(x)$ has a limit as $R \to \infty$, so that 
$w_R(x) \to w_\infty(x)$ for some continuous vector field $w_\infty$.   

Summarizing, we have shown that $u_R(x)$ converges as $R \to \infty$
to some continuous vector field $\bar u(x) = v(x) + w_\infty(x)$ which
satisfies
\[
  |\bar u(x)| \,\le\, C\Bigl(\|\omega\|_{L^\infty} + \|u\|_{L^\infty}\Bigr)
  \log(2+|x|)~, \qquad x \in \R^2~.
\]
By construction, we have $\div \bar u = 0$, $\curl \bar u = \omega$,
and $\bar u(0) = 0$. As in the proof of Proposition~\ref{BSweak}, we
conclude that $u - \bar u$ is identically constant, and this gives
\eqref{BS3}.
 \end{proof}
 
Proposition~\ref{BSultim} shows that the divergence free velocity
field $u \in X$ is entirely determined, up to an additive constant, by
its vorticity distribution $\omega$ even in the case where $\omega$ is
merely bounded. However this result does not provide a good 
reconstruction formula, because the integral in \eqref{BS3} is not 
absolutely convergent. In particular, we cannot use \eqref{BS3} to derive 
an estimate on $\|u\|_{L^\infty}$ in terms of $\|\omega\|_{L^\infty}$, 
but the following consequence of \eqref{BS3} will be useful\:

\begin{corollary}\label{BScor}
Under the assumptions of Proposition~\ref{BSultim}, we have 
for all $x \in \R^2$\:
\begin{equation}\label{BS4}
  u(x) + u(-x) - 2u(0) \,=\,  \frac{1}{2\pi}\int_{\R^2} 
  \biggl\{\frac{(x-y)^\perp}{|x-y|^2} - 
  \frac{(x+y)^\perp}{|x+y|^2} + 2
  \frac{y^\perp}{|y|^2}\biggr\}\omega(y)\dd z~.
\end{equation}
In particular, the following estimate holds
\begin{equation}\label{BS5}
  |u(x) + u(-x) - 2u(0)| \,\le\,  C_* |x| \,\|\omega\|_{L^\infty}~,
  \qquad x \in \R^2~,
\end{equation}
where $C_* > 0$ is a universal constant. 
\end{corollary}

\begin{proof}
Using \eqref{BS3} we easily obtain
\begin{equation}\label{Gaux}
  u(x) + u(-x) - 2u(0) \,=\,  \lim_{R \to \infty}\frac{1}{2\pi}
  \int_{|y|\le R} G(x,y) \omega(y)\dd y~, 
\end{equation}
where
\[
  G(x,y) \,=\, \frac{(x-y)^\perp}{|x-y|^2} - 
  \frac{(x+y)^\perp}{|x+y|^2} + 2 \frac{y^\perp}{|y|^2}~,
  \qquad x \neq \pm y~, \quad y \neq 0~.
\]
A direct calculation yields the bound
\[
  |G(x,y)| \,\le\, C\,\frac{|x|^2}{|x-y|\, |x+y|\, |y|}~,
\]
which implies that
\begin{equation}\label{Gint}
  \int_{\R^2}  |G(x,y)| \dd y \,\le\, C |x|~, \qquad x \in \R^2~,
\end{equation}
for some universal constant $C > 0$. When evaluating that integral, it
is convenient to consider separately the regions where $|y| \le
|x|/2$, where $|x|/2 \le |y| \le 2|x|$, and where $|y| \ge 2|x|$. Thus
we have shown that the integral in \eqref{Gaux} is absolutely
convergent, so that \eqref{BS4} holds, and \eqref{BS5} follows
from \eqref{Gint}. 
\end{proof}

\begin{remark}\label{translation}
The origin plays a distinguished role in formulas \eqref{BS3} and 
\eqref{BS4}, but this is by no mean essential, and more general 
expressions can easily be obtained using translation invariance. 
\end{remark}

\subsection{A representation formula for the pressure}
\label{A2}

Assume that $u \in X$ is such that $\div u = 0$ and $\omega = 
\partial_1 u_2 - \partial_2 u_1 \in L^\infty(\R^2)$. As was discussed
in Section~\ref{ss2.2}, the elliptic equation \eqref{peq} determines a
unique pressure field $p \in \BMO(\R^2)$, up to an irrelevant additive
constant. Setting $q = p + \half |u|^2$ and using identity
\eqref{nlinid}, we obtain for $q$ the equation
\begin{equation}\label{qeq}
  -\Delta q \,=\, \div(u^\perp\omega)~, \qquad x \in \R^2~.
\end{equation}
The goal of this section is to obtain a representation formula
for the solution of \eqref{qeq} involving absolutely convergent 
integrals only, and not singular integrals as in \eqref{pdef}. 

\begin{lemma}\label{qrepresentation}
Assume that $u \in X$, $\div u = 0$ and $\omega = \partial_1 u_2 
- \partial_2 u_1 \in L^\infty(\R^2)$. If $q \in \BMO(\R^2)$ 
is a solution to \eqref{qeq}, we have for any $x_0 \in \R^2$
\begin{equation}\label{qrep}
  q(x) \,=\, q_0 + q_1(x) + q_2(x) + q_3(x,x_0)~, \qquad x 
  \in \R^2~,
\end{equation}
where $q_0 \in \R$ and 
\begin{align}\nonumber
  q_1(x) \,&=\, \frac{1}{2\pi}\int_{\R^2} \chi(x-y) \frac{(x-y)^\perp}{
  |x-y|^2}\cdot u(y)\omega(y)\dd y~, \\ \label{qidef}
  q_2(x) \,&=\, \frac{1}{4\pi}\sum_{k,\ell=1}^2\int_{\R^2} M_{k\ell}(x-y)
  u_k(y)u_\ell(y)\dd y~, \\ \nonumber
  q_3(x,x_0) \,&=\, \frac{1}{2\pi}\sum_{k,\ell=1}^2\int_{\R^2} \Bigl\{
  \chi^c(x-y)K_{k\ell}(x-y) - \chi^c(x_0-y)K_{k\ell}(x_0-y)\Bigr\}
  u_k(y)u_\ell(y)\dd y~. 
\end{align}
Here the following notations have been used\: $\chi \in C_3^\infty(\R^2)$ 
is a cut-off function which is equal to $1$ on a neighborhood of the 
origin, $\chi^c = 1 - \chi$, and
\begin{equation}\label{MKdef}
  M_{k\ell}(z) \,=\, \frac{2z_k \partial_\ell\chi(z) - \delta_{k\ell}
  (z_1\partial_1\chi(z) + z_2\partial_2\chi(z))}{|z|^2}~, \qquad 
  K_{k\ell}(z) \,=\, \frac{2z_kz_\ell - |z|^2\delta_{k\ell}}{|z|^4}~.
\end{equation}
\end{lemma}

\begin{proof}
We first explain how the formulas \eqref{qidef} are obtained. Assume
for simplicity that $\omega$ has compact support in $\R^2$. In that
case, using the fundamental solution of the Laplace operator in $\R^2$
(see Section~\ref{A3}) and integrating by parts, we obtain the unique
solution of \eqref{qeq} which decays to zero at infinity\:
\begin{equation}\label{qfirst}
  q(x) \,=\, -\frac{1}{2\pi}\int_{\R^2} \log|x-y| \div(u^\perp(y)
  \omega(y))\dd y \,=\, \frac{1}{2\pi}\int_{\R^2} \frac{(x-y)^\perp}{
  |x-y|^2}\cdot u(y)\omega(y)\dd y~.
\end{equation}
Our goal is to transform the integral in the right-hand side of 
\eqref{qfirst} into an expression that makes sense even if
$\omega$ is not localized. To do that, we use the partition of 
unity $1 = \chi + \chi^c$ and we decompose $q(x) = q_1(x) + 
\tilde q(x)$, where $q_1$ is given by \eqref{qidef} and 
\begin{equation}\label{qsecond}
  \tilde q(x) \,=\, \frac{1}{2\pi}\int_{\R^2} \chi^c(x-y) 
  \frac{(x-y)^\perp}{|x-y|^2}\cdot u(y)\omega(y)\dd y~. 
\end{equation}
We next invoke the identities
\[
  u_1 \omega \,=\, \partial_1 (u_1 u_2) + \half \partial_2(u_2^2 
  - u_1^2)~, \qquad u_2 \omega \,=\, -\partial_2 (u_1 u_2) 
  + \half \partial_1(u_2^2 - u_1^2)~, 
\]
which allow us to integrate by parts in \eqref{qsecond}. This 
gives two different terms, according to whether the derivative 
acts on the cut-off or on the Biot-Savart kernel. After careful 
calculations, we obtain decomposition $\tilde q(x) = q_2(x) + 
q_3^*(x)$, where $q_2$ is as in \eqref{qidef} and 
\[
  q_3^*(x) \,=\, \frac{1}{2\pi}\sum_{k,\ell=1}^2\int_{\R^2} \chi^c(x-y)
  K_{k\ell}(x-y)u_k(y)u_\ell(y)\dd y~. 
\]
Summarizing, we have $q(x) = q_1(x) + q_2(x) + q_3^*(x)$ when 
$\omega$ is localized, which gives \eqref{qrep} with 
$q_0 = q_3^*(x_0)$ since $q_3(x,x_0) = q_3^*(x) -q_3^*(x_0)$. 

We now consider the general case where $u$ and $\omega$ are 
only supposed to be bounded. Under these assumptions the 
functions $q_1$, $q_2$ defined by \eqref{qidef} are clearly 
continuous and bounded. On the other hand, using the estimate
\begin{equation}\label{Kest}
  |K_{k\ell}(x-y) - K_{k\ell}(x_0-y)| \,\le\, C\biggl(\frac{|x-x_0|}%
  {|x-y|^2 |x_0-y|} + \frac{|x-x_0|}{|x-y|\,|x_0-y|^2}\biggl)~,
\end{equation}
which is obtained as in \eqref{nablaFest}, it is straightforward to
verify that the integral defining $q_3(x,x_0)$ in \eqref{qidef} is
absolutely convergent for any pair of points $x,x_0 \in \R^2$, because
the integrand is bounded and decays to zero like $|y|^{-3}$ as $|y|
\to \infty$. In fact, proceeding as in the proof of
Proposition~\ref{BSultim}, one can show that $q_3(x,x_0)$ is a
continuous function of $x$ which grows at most logarithmically as $|x|
\to \infty$. Now, if we define $\bar q(x) = q_1(x) + q_2(x) +
q_3(x,x_0)$, then $\bar q$ satisfies the elliptic equation
\eqref{qeq}. A convenient way to verify that is to approximate $u$ by
an sequence of compactly supported divergence free vector fields $u_n$
(this can be done using Bogovskii's operator, see \cite{Ga}). The
corresponding pressure $\bar q_n$ satisfies $-\Delta \bar q_n =
\div(u_n^\perp\omega_n)$ by construction, where $\omega_n = \curl
u_n$, and taking the limit $n \to \infty$ we obtain the desired
property for the limit $\bar q$. Finally, if $q \in \BMO(\R^2)$ is any
other solution of \eqref{qeq}, it follows that $q-\bar q$ is a
harmonic function on $\R^2$ with sublinear growth at infinity, hence
$q-\bar q = q_0$ for some $q_0 \in \R$.
\end{proof}

\subsection{A few elementary tools}
\label{A3}

We collect here, for easy reference, a few elementary definitions
and results that are used several times in these notes. 

\medskip\noindent{\bf I. Fourier transforms.} 
We use the following conventions for Fourier transforms on $\R^2$. 
Let $\SS(\R^2)$ denote the (Schwartz) space of all smooth and rapidly 
decreasing functions $f : \R^2 \to \C$, see \cite{RS1}. If 
$f \in \SS(\R^2)$ we set
\begin{equation}\label{FFdef}
  \hat f(\xi) \,=\, \int_{\R^2} f(x) \e^{-i\xi\cdot x}\dd x~, 
  \qquad 
  f(x) \,=\, \frac{1}{2\pi} \int_{\R^2} \hat f(\xi) \e^{i\xi\cdot x}
  \dd \xi~, 
\end{equation}
for all $x \in \R^2$ and $\xi \in \R^2$. We also denote $\hat f = 
\FF f$ and $f = \FF^{-1} f$. According to \eqref{FFdef}, we have 
for any $f \in \SS(\R^2)$\:
\[
  \FF(\nabla_x f) \,=\, i\xi(\FF f)~, \qquad \hbox{and}\quad
  \FF(xf) \,=\, i\nabla_\xi (\FF f)~.
\]
The Fourier transform $\FF$ and its inverse $\FF^{-1}$ are linear
isomorphisms on $\SS(\R^2)$, and can be extended to linear
isomorphisms on the dual space $\SS'(\R^2)$, which is the space of
tempered distributions on $\R^2$ \cite{RS1,RS2}. For instance, 
if $\delta_0$ denotes the Dirac measure located at the origin, 
we have $(\FF \delta_0)(\xi) = 1$ for all $\xi \in \R^2$. 

\medskip\noindent{\bf II. Young's inequality.} 
If $f,g \in \SS(\R^2)$, we define the convolution product 
$h = f*g \in \SS(\R^2)$ by the formula
\begin{equation}\label{convdef}
  h(x) \,=\, \int_{\R^2} f(x-y)g(y)\dd y \,=\, 
  \int_{\R^2} g(x-y)f(y)\dd y~, \qquad x \in \R^2~.
\end{equation}
In Fourier space, we then have $\hat h(\xi) = \hat f(\xi)\hat g(\xi)$ 
for all $\xi \in \R^2$, so that $\FF(f*g) = (\FF f)(\FF g)$. 
Moreover, for all exponents $p,q,r \in [1,\infty]$ satisfying
$\frac1p + \frac1q =1 + \frac1r$, we have Young's inequality
\begin{equation}\label{Youngdef}
  \|h\|_{L^r(\R^2)} \,=\, \|f*g\|_{L^r(\R^2)} \,\le\, \|f\|_{L^p(\R^2)} 
  \|g\|_{L^q(\R^2)}~. 
\end{equation}
More generally, if $f \in L^p(\R^2)$ and $g \in L^q(\R^2)$, one 
can show that the integral in \eqref{convdef} converges for almost 
every $x \in \R^2$ and defines a function $h \in L^r(\R^2)$ 
satisfying \eqref{Youngdef}. 

\medskip\noindent{\bf III. Fundamental solutions.} The Fourier
transform can be used to compute fundamental solutions of 
partial differential operators with constant coefficients. 
Two particular examples play an important role in these notes. 
First, the fundamental solution of the Poisson equation
$\Delta \Phi = \delta_0$ in $\R^2$ is
\[
  \Phi(x) \,=\, \frac{1}{2\pi}\,\log|x|~, \qquad 
  x \in \R^2\setminus\{0\}~,
\]
see \cite{Ev,LL}. It follows that $u = \Phi * \rho$ is the solution 
of the Poisson equation $\Delta u = \rho$ for any $\rho \in \SS(\R^2)$. 
Similarly, the vector field
\[
  V(x) \,=\, \frac{1}{2\pi}\,\frac{x^\perp}{|x|^2}~, 
  \qquad x \in \R^2\setminus\{0\}~,
\]
satisfies $\div V = 0$ and $\curl V \equiv \partial_1 V_2 - 
\partial_2 V_1 = \delta_0$. Thus, if $u = V * \omega$ for some $\omega
\in \SS(\R^2)$, we have $\div u = 0$ and $\curl u = \omega$. The
vector field $V = \nabla^\perp \Phi$ is therefore the fundamental 
solution associated to the Biot-Savart law.

\medskip\noindent{\bf IV. Gronwall's lemma.} 
There exist many versions of Gronwall's lemma, but the following
one is sufficient for our purposes. 

\begin{lemma}\label{Gronwall}
Let $T > 0$, $a \ge 0$, and assume that $f,g,b : [0,T] \to 
\R_+$ are continuous functions satisfying
\begin{equation}\label{G1}
  f(t) + \int_0^t g(s)\dd s \,\le\, a + \int_0^t b(s)f(s)\dd s~,
  \qquad 0 \le t \le T~.
\end{equation}
Then
\begin{equation}\label{G2}
  f(t) + \int_0^t g(s)\dd s \,\le\, a\exp\Bigl(\int_0^t b(s)\dd s
  \Bigr)~, \qquad  0 \le t \le T~.
\end{equation}
\end{lemma}

\begin{proof}
Let $F(t) = \int_0^t b(s)f(s)\dd s$. Then $F$ is continuously 
differentiable on $[0,T]$ and satisfies, in view of 
\eqref{G1}, 
\[
  F'(t) = b(t)f(t) \,\le\, ab(t) + b(t)F(t)~, \qquad 
  0 \le t \le T~.
\]
Integrating that differential inequality and observing 
that $F(0) = 0$, we obtain the bound
\[
  F(t) \,\le\, a\exp\Bigl(\int_0^t b(s)\dd s\Bigr) - a~, 
  \qquad  0 \le t \le T~,
\]
which can be inserted in the right-hand side of \eqref{G1} 
to give \eqref{G2}. 
\end{proof}


\begin{thebibliography}{99}
\setlength{\itemsep}{-0.4mm}

\bibitem{AM} A. Afendikov and A. Mielke, Dynamical properties of
spatially non-decaying 2D Navier-Stokes flows with Kolmogorov 
forcing in an infinite strip, J. Math. Fluid. Mech. \textbf{7} 
(2005), suppl. 1, S51--S67.

\bibitem{AKLN} D. Ambrose, J. Kelliher, M. Lopes Filho, and H. 
Nussenzveig Lopes, Serfati solutions to the 2D Euler equations 
on exterior domains, preprint {\tt arXiv:1401.2655}.

\bibitem{ARCD} J. Arrieta, A. Rodriguez-Bernal, J. Cholewa, and
T. Dlotko, Linear parabolic equations in locally uniform
spaces, Math. Models Methods Appl. Sci.  \textbf{14} (2004),
253--293.

\bibitem{BL}
J. Bergh and J. L\"ofstr\"om, {\em Interpolation spaces. An 
introduction}, Grundlehren der Mathematischen Wissenschaften 
\textbf{223}, Springer, 1976. 

\bibitem{CZ}
V. Chepyzhov and S. Zelik, 
Infinite-energy solutions for dissipative Euler equations in
$\R^2$, in preparation.

\bibitem{CF} P. Constantin and C. Foias, 
{\em Navier-Stokes equations}, Chicago Lectures in Mathematics, 
University of Chicago Press, 1988. 

\bibitem{EZ} M. Efendiev and S. Zelik, 
The attractor for a nonlinear reaction-diffusion system 
in an unbounded domain, Comm. Pure Appl. Math. {\bf 54} (2001),  
625--688. 

\bibitem{EN} K.-J. Engel and R. Nagel, {\em One-parameter semigroups 
for linear evolution equations}, Graduate Texts in Mathematics 
{\bf 194}, Springer, 2000.

\bibitem{Ev} L. Evans, 
{\em Partial differential equations}, Graduate Studies in Mathematics
{\bf 19}, American Mathematical Society, Providence, RI, 1998.

\bibitem{Fe} E. Feireisl, Bounded, locally compact global attractors 
for semilinear damped wave equations on $\R^n$, J. Diff. Integral Eqns. 
{\bf 9} (1996), 1147--1156.

\bibitem{FK} H.~Fujita and T.~Kato,
On the Navier-Stokes initial value problem. I,
Arch. Rational Mech. Anal. {\bf 16} (1964), 269--315. 

\bibitem{Ga} G. Galdi, 
{\em An introduction to the mathematical theory of the Navier-Stokes 
equations. Steady-state problems}, Springer Monographs in 
Mathematics, Springer, 2011.

\bibitem{GS0} Th. Gallay and S. Slijep\v cevi\'c, 
Energy flow in formally gradient partial differential equations on 
unbounded domains. J. Dynam. Differential Equations {\bf 13} (2001), 
757--789. 

\bibitem{GS1} Th. Gallay and S. Slijep\v{c}evi\'{c}, Energy bounds 
for the two-dimensional Navier-Stokes equations in an infinite
cylinder, Comm. Partial Differential Equations {\bf 39} (2014), 
1741--1769.  

\bibitem{GS2} Th. Gallay and S. Slijep\v{c}evi\'{c}, 
Uniform boundedness and long-time asymptotics
for the two-dimensional Navier-Stokes equations in an 
infinite cylinder, to appear in JDDE. 

\bibitem{GW} Th. Gallay and C.E. Wayne, 
Global stability of vortex solutions of the two-dimensional 
Navier-Stokes equation, Comm. Math. Phys. {\bf 255} (2005), 97--129. 

\bibitem{GIM} Y. Giga, Yoshikazu, K. Inui, and S. Matsui, 
On the Cauchy problem for the Navier-Stokes equations with 
nondecaying initial data, in {\em Advances in fluid dynamics},
Quaderni di Mathematica {\bf 4} (1999), 27--68.

\bibitem{GMS} Y. Giga, S. Matsui, and O. Sawada, Global existence 
of two dimensional Navier-Stokes flow with non-decaying initial velocity, 
J. Math. Fluid. Mech. \textbf{3} (2001), 302--315.

\bibitem{GV} J. Ginibre and G. Velo,
The Cauchy problem in local spaces for the complex Ginzburg-Landau 
equation, I. Compactness methods, Physica D {\bf 95} (1996), 191--228. 
II. Contraction methods, Comm. Math. Phys. {\bf 187} (1997), 45--79.

\bibitem{He} D. Henry,
{\em Geometric theory of semilinear parabolic equations}, 
Lecture Notes in Mathematics {\bf 840}, Springer, 1981.

\bibitem{Ka} T. Kato,
The Cauchy problem for quasi-linear symmetric hyperbolic systems, 
Arch. Rational Mech. Anal. {\bf 58} (1975), 181--205. 

\bibitem{Ke} J. Kelliher, 
A characterization at infinity of bounded vorticity, bounded
velocity solutions to the 2D Euler equations, 
preprint {\tt arXiv:1404.3404}.

\bibitem{KO}
H.~Kozono and T.~Ogawa,
Two-dimensional Navier-Stokes flow in unbounded domains,
Math. Ann. {\bf 297} (1993), 1--31. 

\bibitem{La}
O.~A. Lady{\v{z}}enskaja, 
Solution ``in the large'' of the nonstationary boundary value problem
for the Navier-Stokes system with two space variables, 
Comm. Pure Appl. Math. {\bf 12} (1959), 427--433. 

\bibitem{Le1}
J.~Leray,
\'Etude de diverses \'equations int\'egrales non lin\'eaires et de
quelques probl\`emes que pose l'hydrodynamique,
J. Math. Pures Appliqu\'ees 9\`eme s\'erie {\bf 12} (1933), 1--82. 

\bibitem{Le2}
J.~Leray,
Essai sur les mouvements plans d'un fluide visqueux que limitent 
des parois, 
J. Math. Pures Appliqu\'ees 9\`eme s\'erie {\bf 13} (1934), 331--418. 

\bibitem{LL} E. Lieb and M. Loss, {\em Analysis}, 
Graduate Studies in Mathematics  {\bf 14}, American Mathematical 
Society, Providence, RI, 1997.

\bibitem{Li} P.-L. Lions, 
{\em Mathematical topics in fluid mechanics. Vol. 1. Incompressible
models}, Oxford Lecture Series in Mathematics and its Applications {\bf 3}, 
Oxford University Press, 1996.

\bibitem{MB} A. Majda and A. Bertozzi,
{\em Vorticity and incompressible flow}, 
Cambridge Texts in Applied Mathematics {\bf 27}, 
Cambridge University Press, Cambridge, 2002. 

\bibitem{Ma}
K.~Masuda, 
Weak solutions of Navier-Stokes equations,
Tohoku Math. J. (2) {\bf 36} (1984), 623--646. 

\bibitem{MT} Y. Maekawa and Y. Terasawa,
The Navier-Stokes equations with initial data in uniformly local 
$L^p$ spaces, Diff. Int. Equations {\bf 19} (2006), 369--400.

\bibitem{MS} A. Mielke and G. Schneider, 
Attractors for modulation equations on unbounded domains --- existence 
and comparison, Nonlinearity {\bf 8} (1995), 743--768.

\bibitem{Me} Y. Meyer, 
{\em Ondelettes et op\'erateurs. II. Op\'erateurs de 
Calder\'on-Zygmund}, Actualit\'es Math\'ematiques, Hermann, 
Paris, 1990.

\bibitem{Pa} A. Pazy, {\em Semigroups of linear operators and 
applications to partial differential equations}, 
Applied Mathematical Sciences {\bf 44}, Springer, New York, 1983.

\bibitem{PW} M. Protter and H. Weinberger, 
{\em Maximum principles in differential equations}, Prentice-Hall, 
Englewood Cliffs, N.J., 1967. 

\bibitem{RS1} M. Reed and B. Simon,
{\em Methods of modern mathematical physics. I. Functional analysis}, 
Academic Press, New York, 1972.

\bibitem{RS2} M. Reed and B. Simon,
{\em Methods of modern mathematical physics. II. Fourier analysis, 
self-adjointness}, Academic Press, New York, 1975.

\bibitem{ST} O. Sawada and Y. Taniuchi, A remark on $L^{\infty }$
solutions to the 2-D Navier-Stokes equations, J. Math. Fluid Mech. 
\textbf{9} (2007), 533--542.

\bibitem{Sch} M. Schonbek, 
Large time behaviour of solutions to the Navier-Stokes equations,
Comm. Partial Differential Equations {\bf 11} (1986), 733--763. 

\bibitem{Se} Ph. Serfati, Solutions $C^\infty$ en temps, $n$-log Lipschitz 
born\'ees en espace et \'equation d'Euler, C. R. Acad. Sci. Paris 
S\'er. I Math. {\bf 320} (1995), 555--558. 

\bibitem{St} E. Stein, {\em Harmonic analysis: real-variable methods, 
orthogonality, and oscillatory integrals}, Princeton Mathematical 
Series {\bf 43}, Princeton, 1993.

\bibitem{Te} R. Temam,
{\em Navier-Stokes equations. Theory and numerical analysis. 
Third edition}, Studies in Mathematics and its Applications 
{\bf 2}, North-Holland, Amsterdam, 1984. 

\bibitem{Wi} M. Wiegner, 
Decay results for weak solutions of the Navier-Stokes equations on 
$\R^n$, J. London Math. Soc. (2) {\bf 35} (1987), 303--313. 

\bibitem{Ze} S. Zelik, Infinite energy solutions for damped
Navier-Stokes equations in $\R^2$, J. Math. Fluid Mech. \textbf{15}
(2013), 717--745.

\end{thebibliography}
\end{document}